\documentclass[11pt]{amsart}

\usepackage{fullpage}

\usepackage{amsthm, amsmath, amssymb, mathtools}
\usepackage[initials]{amsrefs}
\usepackage[hidelinks]{hyperref}
\usepackage[mathscr]{eucal}
\usepackage{enumerate}

\usepackage{graphicx}
\usepackage{caption}

\allowdisplaybreaks

\numberwithin{equation}{section}

\newtheorem{prop}[equation]{Proposition}
\newtheorem{thm}[equation]{Theorem}

\newcommand{\R}{\mathbb{R}}
\newcommand{\C}{\mathbb{C}}
\newcommand{\Z}{\mathbb{Z}}

\newcommand{\Sph}{\mathbb{S}}
\newcommand{\T}{\mathbb{T}}

\renewcommand{\Re}{\operatorname{Re}}
\renewcommand{\Im}{\operatorname{Im}}
\providecommand{\ii}{{\mathrm{i}}}

\providecommand{\im}{\operatorname{im}}

\providecommand{\Tcol}{\mathbf{T}}

\providecommand{\acirc}[1]{C_{(#1)}}

\providecommand{\acol}{\mathbf{C}}
\providecommand{\rcol}{\mathbf{\Gamma}}
\providecommand{\scol}{\mathbf{\Sigma}}
\providecommand{\scaf}{\mathbf{L}}
\providecommand{\ptcol}{\mathbf{p}}
\providecommand{\prcol}{\mathbf{\Omega}}
\providecommand{\ctrcol}{\boldsymbol{\omega}}

\providecommand{\phopf}{\Pi_{\mathrm{Hopf}}}
\providecommand{\pghopf}{\widehat{\Pi}_{\mathrm{Hopf}}}

\providecommand{\rot}{\mathsf{R}}
\providecommand{\refl}{{\underline{\mathsf{R}}}}
\providecommand{\grp}{\mathscr{G}}
\providecommand{\subgrp}{\mathscr{H}}
\providecommand{\stab}{\operatorname{Stab}}

\providecommand{\abs}[1]{\left\lvert#1\right\rvert}

\title{
New low-genus desingularizations of three Clifford tori
and related characterizations
}

\author{Nikolaos Kapouleas}
\address{
  Department of Mathematics,
  Brown University,
  Providence, RI 02912
}
\email{nicolaos\_kapouleas@brown.edu}

\author{David Wiygul}
\address{
  Dipartimento di Matematica,
  Universit\`{a} degli studi di Trento,
  38123 Povo TN, Italy
}
\email{davidjames.wiygul@unitn.it}

\begin{document}

\date{\today}

\begin{abstract}
For each nonnegative integer $m$
we construct in the round three-sphere
a closed embedded minimal surface
of genus $48m+25$
which can be interpreted
as a desingularization of the union
of three Clifford tori
intersecting pairwise orthogonally,
along a total of six great circles.
Each such surface is generated,
under the action of a group of symmetries,
by a disc with hexagonal boundary,
all of whose sides are contained in great circles.
We prove a uniqueness result for this disc,
and, as a corollary,
we characterize these surfaces.
This characterization implies
that similar surfaces we constructed
for sufficiently high $m$
by gluing methods,
in an earlier article,
coincide with the ones here.
For low $m$ the surfaces constructed here are new.
Similarly, we prove uniqueness
of the generating discs
for one of two families  
constructed by Choe and Soret
(namely the surfaces they call odd)
and show that these surfaces also coincide,
when of sufficiently high genus,
with surfaces we have constructed by gluing.
\end{abstract}

\maketitle

\section{Introduction}
In \cite{ChoeSoretTordesing}
J. Choe and M. Soret
constructed in the round $3$-sphere $\Sph^3$
for each integer $k \geq 2$
two families,
$\{T^o_{k,2km}\}_{m=1}^\infty$
and
$\{T^e_{k,2km}\}_{m=2}^\infty$,
of closed embedded minimal surfaces
which can be regarded as desingularizations
of the union of a maximally symmetric configuration
$\Tcol[k]$
of $k$ Clifford tori
intersecting transversely
along a pair of great circles
at distance $\pi/2$.
The surfaces
$T^o_{k,2km}$ and $T^e_{k,2km}$
both have genus
$1+4(k-1)km$.
Choe and Soret call these surfaces, respectively,
odd and even,
in view of their symmetries.
The construction of the odd surfaces
is reminiscent of Lawson's tiling construction
\cite{Lawson};
namely,
Choe and Soret solve a certain Plateau problem
with piecewise geodesic boundary in $\Sph^3$,
and they then extend the solution
to a closed embedded minimal surface
by applying the group generated by the reflections
through the great circles containing the boundary
of the solution.
The even surfaces
are instead constructed
by applying a different group of symmetries
to the solution
to a partially free boundary value problem,
subject to a certain symmetry condition,
but we will not discuss them further in this article.

In \cite{KWtordesing}
we applied gluing methods
to construct many closed embedded minimal surfaces
admitting the same interpretation
(as desingularizations of
$\bigcup \Tcol[k]$)
and also examples which desingularize rather
the union of the collection
$\Tcol[k,1]$
consisting of the Clifford tori
in the preceding collection $\Tcol[k]$
together with the additional Clifford torus
equidistant from the original two circles of intersection.
(This last union $\bigcup\Tcol[k,1]$
then has $2k+2$ great circles
of intersection.)
Since the surfaces of \cite{KWtordesing}
are constructed by gluing,
all have high genus.

In the present article we adapt
the Plateau tiling method
used in \cite{Lawson}
and, for the odd surfaces, in \cite{ChoeSoretTordesing}
to construct further desingularizations
of $\bigcup \Tcol[2,1]$,
a union of three Clifford tori
intersecting pairwise orthogonally
along six great circles.
We also adapt arguments
from \cite{KWlindex}
and \cite{KWlchar},
there used to establish uniqueness
and a certain graphicality of the Plateau solutions
generating the Lawson surfaces
(which can be regarded
as desingularizations of intersecting spheres),
in order to establish analogous results,
namely Theorem \ref{unique_disc},
for solutions to Plateau problems
generating desingularizations
of intersecting Clifford tori.
As one application of Theorem \ref{unique_disc}
we obtain a characterization
of the new desingularizations of
$\bigcup \Tcol[2,1]$
which in particular suffices
to confirm that they agree,
when of sufficiently high genus,
with the members of a certain subfamily
of the surfaces we constructed in
\cite{KWtordesing}.

\begin{thm}[Desingularizations
  of three Clifford tori
  intersecting pairwise orthogonally]
\label{thm:three_tori}
For each integer $m \geq 0$
there exists in $\Sph^3$
a unique closed embedded minimal surface
$M_m$
whose intersection with the interior
of the prism $\Omega_{00}$
(defined below \ref{centers}
and described in \ref{prisms_for_three_tori})
is nonempty and satisfies
\begin{equation*}
  \partial (M_m \cap \Omega_{00})
  =
  \Omega_{00} \cap \bigcup (\rcol \cup \scaf^{4+8m})
\end{equation*}
(where $\rcol$ and $\scaf^{4+8m}$
are collections of great circles
specified in
\ref{rcol} and \ref{scaf}).
Furthermore,
$M_m$ has genus $25+48m$,
contains every great circle in
$\rcol \cup \scaf^{4+8m}$,
and is invariant under the group
$\grp^{4+8m}$
(defined in \ref{group_definition}
and described in \ref{group_description}).
Finally, in view of the uniqueness assertion,
for all sufficiently large $m$
the surface $M_m$ is congruent to
the surface in \cite{KWtordesing}
obtained
(as in \cite{KWtordesing}*{Theorem 7.1})
from the initial surface
$N(2, 2m+1, 1, 1, 1, 0, 1)$
(as defined in \cite{KWtordesing}*{Definition 4.13}).
\end{thm}

Section \ref{sec:three_tori}
is devoted to the proof of
Theorem \ref{thm:three_tori}.

As another application of Theorem \ref{unique_disc},
we obtain also a characterization
of the odd surfaces of Choe and Soret,
which in particular confirms that they agree,
when of sufficiently high genus,
with the members of a certain subfamily
we constructed in \cite{KWtordesing}.

\begin{thm}[A characterization
  of the odd surfaces of Choe and Soret
  \cite{ChoeSoretTordesing}]
\label{odd_cs_uniqueness}
For all integers $k \geq 2$ and $m \geq 1$
there is in $\Sph^3$
a unique closed embedded minimal surface
whose intersection with the interior of
$\Omega_{\frac{\pi}{4},\frac{\pi}{k},\frac{\pi}{2km}}$
(as defined in \ref{prism_def}
and described in \ref{prism_geometry})
is nonempty and has boundary the hexagon
consisting of all the geodesic edges of
$\Omega_{\frac{\pi}{4},\frac{\pi}{k},\frac{\pi}{2km}}$
except the one along $\acirc{1}$;
this surface is 
congruent to $T^o_{k,2km}$
(as constructed in \cite{ChoeSoretTordesing}*{Theorem 1}).
In particular for all $m$ sufficiently large
$T^o_{k,2km}$ is congruent to
the surface in \cite{KWtordesing}
obtained
(as in \cite{KWtordesing}*{Theorem 7.1})
from the initial surface
$M(k,2m,1,1,0)$
(as defined in \cite{KWtordesing}*{Definition 4.13}).
\end{thm}

The proof of Theorem \ref{odd_cs_uniqueness}
is given at the end of Section \ref{sec:uniqueness}.

\subsection*{Acknowledgments}
This project has received funding
from the European Research Council (ERC)
under the European
Union’s Horizon 2020 research and innovation programme
(grant agreement No. 947923).
Specifically, DW was supported by the preceding grant
and thanks Alessandro Carlotto for making
that funding possible and for his interest
in this line of research.

NK would like to thank
the Simons Foundation
for Collaboration Grant 962205,
which facilitated travel related to this research.

\section{Notation and conventions}
Given $X \subseteq \Sph^3$,
we write $d^{\Sph^3}_X$ for the function assigning
to each point in $\Sph^3$ its distance to $X$.
Given a great circle $C$,
we write $C^\perp$ for the unique great circle
on which $d^{\Sph^3}_C$ is constantly $\pi/2$.
Given $t \in \R$
together with a great circle $C$
and an orientation on $C^\perp$,
we write $\rot_C^t$ for the isometry of $\Sph^3$
fixing $C$ identically and rotating $C^\perp$
along itself through angle $t$;
we write $\refl_C:=\rot_C^\pi$
for reflection through $C$
(well-defined
even in the absence
of an orientation on $C^\perp$).
If $\{\mathsf{T}_\alpha\}_{\alpha \in A}$
is a collection of a isometries of $\Sph^3$
(elements of $O(4)$),
we write
$\langle \mathsf{T}_\alpha \rangle_{\alpha \in A}$
for the subgroup of $O(4)$ they generate.

We will make the identifications
\begin{equation}
\begin{gathered}
  \Sph^3 
  :=
  \{
    (w_1,w_2) \in \C^2
    \; : \;
    \abs{w_1}^2+\abs{w_2}^2 = 1
  \},
  \\
  \Sph^2(1/2)
  :=
  \left\{
    (x,y,z) \in \R^3
    \; : \;
    x^2+y^2+z^2=1/4
  \right\},
\end{gathered}
\end{equation}
as sets,
and we define the Hopf projection
\begin{equation}
\begin{gathered}
  \phopf: \Sph^3 \to \Sph^2(1/2)
  \\
  (w_1,w_2)
  \mapsto
  \left(
    \Re w_1\overline{w_2}, ~
    \Im w_1\overline{w_2}, ~
    \frac{\abs{w_1}^2-\abs{w_2}^2}{2}
  \right),
\end{gathered}
\end{equation}
a Riemannian submersion,
with the usual round metrics on the domain and target.

Throughout the article we distinguish the great circle
\begin{equation}
  \acirc{1} := \{(e^{\ii s},0) \in \Sph^3 \; : \; s \in \R\},
\end{equation}
so that
\begin{equation}
  \acirc{1}^\perp := \{(0,e^{\ii t}) \in \Sph^3 \; : \; t \in \R\},
\end{equation}
and we orient $\acirc{1}$ and $\acirc{1}^\perp$ so that
\begin{equation}
  \rot_C^s(w_1,w_2)
    =
    (w_1, e^{\ii s}w_2),
  \quad
  \rot_{C^\perp}^t(w_1,w_2)
    =
    (e^{\ii t}w_1,w_2).
\end{equation}
The fibers of $\phopf$
are then the orbits of
$\{\rot_{\acirc{1}}^t\rot_{\acirc{1}^\perp}^t\}_{t \in \R}$,
and,
orienting each orbit in the direction of increasing $t$,
we have
\begin{equation*}
  \rot_C^t\rot_{C^\perp}^t
  =
  \rot_{\acirc{1}}^t\rot_{\acirc{1}^\perp}^t
\end{equation*}
for any orbit $C$ and any $t \in \R$.

We denote the sets of unitary
and antiunitary transformations on $\C^2$
by, respectively,
\begin{equation}
\begin{aligned}
  U(2)
    &:=
    \{
      \mathsf{T}: \C^2 \to \C^2
      \; : \;
      \forall x,y \in \C^2 \;\;
      \langle \mathsf{T}x, \mathsf{T}y \rangle_{\C^2}
        =
        \langle x,y \rangle_{\C^2}
    \},
  \\
  \overline{U}(2)
    &:=
    \{
      \mathsf{T}: \C^2 \to \C^2
      \; : \;
      \forall x,y \in \C^2 \;\;
      \langle \mathsf{T}x, \mathsf{T}y \rangle_{\C^2}
        =
        \overline{\langle x,y \rangle}_{\C^2}
    \}
\end{aligned}
\end{equation}
(where $\langle \cdot,\cdot\cdot \rangle_{\C^2}$
is the standard (Hermitian) inner product on $\C^2$)
and we identify them
as subsets of $O(4)$ in the obvious way.
(Then $U(2)$ and $U(2) \cup \overline{U}(2)$
are in fact subgroups of $O(4)$.)
We observe that $\rot \in O(4)$
preserves the collection of fibers of $\phopf$
if and only if
it is unitary or antiunitary,
and we define the map
\begin{equation}
\label{pghopf_declaration}
  \pghopf: U(2) \cup \overline{U}(2) \to O(3)
\end{equation}
(the target being the group of orthogonal
transformations on $\R^3$)
which takes any given
$\rot \in U(2) \cup \overline{U}(2)$
to the unique element $\pghopf \rot$ of $O(3)$
satisfying
\begin{equation}
\label{pghopf_definition}
  \phopf \circ \rot
  =
  (\pghopf \rot) \circ \phopf.
\end{equation}
Then $\pghopf$ is a homomorphism.
Note that
$U(2)$ preserves the orientations
of the fibers of $\phopf$,
while $\overline{U}(2)$ reverses them,
and that
$\pghopf(U(2))=SO(3)$
(the group of orientation-preserving
orthogonal transformations on $\R^3$).

\section{Uniqueness and graphicality 
of solutions to the Plateau problems}
\label{sec:uniqueness}

We start with some definitions.
First, we set
\begin{equation}
  C^0_0 := \{(\cos r, \sin r) \; : \; r \in \R\},
\end{equation}
the great circle orthogonally intersecting
$\acirc{1}$ at $(\pm 1,0)$
and $\acirc{1}^\perp$ at $(0,\pm 1)$,
and we set
\begin{equation}
  C^{\pi/2}_{\pi/2}
  :=
  (C^0_0)^\perp
  =
  \{(i \cos r, i \sin r) \; : \; r \in \R\}.
\end{equation}
Next, for each $z \in \R$ we define
\begin{equation}
  \Sigma_z
  :=
  \{
    (e^{\ii z}\cos r, e^{\ii\theta} \sin r) \in \Sph^3
    \; : \;
    r,\theta \in \R
  \},
\end{equation}
the great sphere containing $\acirc{1}^\perp$
and intersecting $\acirc{1}$ orthogonally
at $(\pm e^{\ii z},0)$,
and for each $\theta \in \R$ we define also
\begin{equation}
  \Sigma^\theta
  :=
  \{
    (e^{\ii z}\cos r, e^{\ii\theta} \sin r) \in \Sph^3
    \; : \;
    r,z \in \R
  \},
\end{equation}
the great sphere containing $\acirc{1}$
and intersecting $\acirc{1}^\perp$ orthogonally
at $(0,\pm e^{\ii\theta})$,
and
\begin{equation}
  \T^\theta
  :=
  \left\{
    e^{\ii z}(\cos r, e^{\ii\theta} \sin r) \in \Sph^3
    \; : \;
    r,z \in \R
  \right\},
\end{equation}
a Clifford torus (one of two)
containing
$\acirc{1} \cup \rot_{\acirc{1}}^\theta C^0_0$
(and $\acirc{1}^\perp \cup \rot_{\acirc{1}}^\theta C^{\pi/2}_{\pi/2}$).
Given $r \in (0,\pi/2)$, we define
\begin{equation}
  T^r := \{d^{\Sph^3}_{\acirc{1}}=r\}
  =
  \{
    e^{iz}(\cos r, e^{i\theta}\sin r)
    \; : \;
    \theta,z \in \R
  \},
\end{equation}
a constant-mean-curvature torus.
Finally, given $S \in (0,\pi/4]$, $\Theta \in (0,\pi/2]$,
and $Z \in (0,\pi/4]$,
we define 
\begin{equation}
\label{prism_def}
  \Omega_{S,\Theta,Z} 
  :=
  \left\{
    e^{\ii z}(\cos r, e^{\ii\theta} \sin r) \in \Sph^3
    \; : \;
    r \in [0,S]
    \mbox{ and }
    \theta \in
      \left[
        -\frac{\Theta}{2},\frac{\Theta}{2}
      \right]
    \mbox{ and }
    z \in
      \left[
        -\frac{Z}{2},\frac{Z}{2}
      \right]
  \right\},
\end{equation}
a sphericotoral pentahedron,
or a prism with base an isosceles triangle,
as described presently in greater detail.
As all claims of the below proposition
are simple observations,
we do not include a proof.

\begin{prop}[Geometry of $\Omega_{S,\Theta,Z}$]
\label{prism_geometry}
Let $S \in (0,\pi/4]$, $\Theta \in (0,\pi/2]$,
and $Z \in (0,\pi/4]$.
Then $\Omega_{S,\Theta,Z}$
is homeomorphic
to a closed $3$-ball
and has the following properties:
\begin{enumerate}[(i)]
\item 
$
  \partial\Omega_{S,\Theta,Z}
  =
  R_{S,\Theta,Z}^- \cup R_{S,\Theta,Z}^+
   \cup T_{S,\Theta,Z}^- \cup T_{S,\Theta,Z}^+
   \cup P_{S,\Theta,Z}
$,
where $R_{S,\Theta,Z}^{\pm}$ are solid rectangles
lying on the two Clifford tori
$\T^{\pm \Theta/2}$,
$T_{S,\Theta,Z}^{\pm}$ are solid triangles
lying on the two great spheres
$\Sigma_{\pm Z/2}$,
and $P_{S,\Theta,Z}$ is a solid parallelogram
lying on the constant-mean-curvature torus
$T^S$;

\item $R_{S,\Theta,Z}^\pm$ has sides
      $R_{S,\Theta,Z}^\pm \cap T_{S,\Theta,Z}^+$,
      $R_{S,\Theta,Z}^\pm \cap T_{S,\Theta,Z}^-$,
      $
        R_{S,\Theta,Z}^\pm \cap \acirc{1}
        =
        R_{S,\Theta,Z}^+ \cap R_{S,\Theta,Z}^-
      $,
      and $R_{S,\Theta,Z}^\pm \cap P_{S,\Theta,Z}$,
      and vertex angles all of measure $\pi/2$;

\item $T_{S,\Theta,Z}^\pm$ has sides 
      $T_{S,\Theta,Z}^\pm \cap R_{S,\Theta,Z}^+$,
      $T_{S,\Theta,Z}^\pm \cap R_{S,\Theta,Z}^-$,
      and $T_{S,\Theta,Z}^\pm \cap P_{S,\Theta,Z}$,
      the first two meeting on $\acirc{1}$ at angle $\Theta$
      and each of the remaining two vertex angles
      having measure $\pi/2$;

\item $P_{S,\Theta,Z}$ has sides
      $P_{S,\Theta,Z} \cap T_{S,\Theta,Z}^+$,
      $P_{S,\Theta,Z} \cap T_{S,\Theta,Z}^-$,
      $P_{S,\Theta,Z} \cap R_{S,\Theta,Z}^+$,
      and $P_{S,\Theta,Z} \cap R_{S,\Theta,Z}^-$,
      and vertex angles of measure
      $\pi/2 \pm S$; 

\item $
        R_{S,\Theta,Z}^\pm \cap \acirc{1}
        =
        R_{S,\Theta,Z}^+ \cap R_{S,\Theta,Z}^-
      $
      has constant dihedral angle $\Theta$
      and is an arc of length $Z$
      on the great circle $\acirc{1}$;
      
\item each of the four edges
      $R_{S,\Theta,Z}^{{\pm}_1} \cap T_{S,\Theta,Z}^{{\pm}_2}$ 
      has dihedral angle varying monotonically
      between $\pi/2$ and $\pi/2 \pm S$
      and is an arc of length $S$
      on a great circle
      orthogonally intersecting $\acirc{1}$;

\item each of the two edges
      $P_{S,\Theta,Z} \cap R_{S,\Theta,Z}^\pm$
      has constant dihedral angle $\pi/2$
      and is an arc of length $Z$
      on a great circle on $T^S$;
      
\item each of the two edges
      $P_{S,\Theta,Z} \cap T_{S,\Theta,Z}^\pm$
      has constant dihedral angle $\pi/2$
      and is an arc of length
      $\Theta \sin S$,
      simultaneously a circle of latitude
      on $\Sigma_{\pm Z/2}$
      and a circle of principal curvature
      on $T^S$;
    
\item  
$
  \refl_{C^0_0}\Omega_{S,\Theta,Z}
  =
  \Omega_{S,\Theta,Z}
$; and

\item $\Omega_{S,\Theta,Z}$ is contained in the closure
      of a single component of 
      $
        \Sph^3
        \setminus
        (
          \Sigma_{\pi/2}
          \cup
          \Sigma^{\pi/2}
        )
      $.
\end{enumerate}
\end{prop}

Next we consider the one-parameter group
$\Bigl\{\rot_{C^{\pi/2}_{\pi/2}}^t\Bigr\}_{t \in \R}$
of rotations along $C^0_0$,
and we define on $\Sph^3$ the Killing field $K$
by
\begin{equation}
  K|_{(w_1,w_2)}
  :=
  \left.\frac{d}{dt}\right|_{t=0}
    \rot_{C^{\pi/2}_{\pi/2}}^t
    (w_1,w_2)
  =
  (-\Re w_2, ~ \Re w_1).
\end{equation}
Whenever we refer to an orbit of $K$,
it will always be oriented in the direction of $K$.
We call
$p \in \partial\Omega_{S,\Theta,Z}$
an entry or exit point of $\Omega_{S,\Theta,Z}$ for
$\pm K$
(and say that the $\pm K$ orbit of $p$ enters or exits
$\Omega_{S,\Theta,Z}$ at $p$)
if there exists $\epsilon>0$
such that the set
$\{\rot_{C^{\pi/2}_{\pi/2}}^{\pm t} p\}_{0<t<\epsilon}$
is contained in $\Omega_{S,\Theta,Z}$
or, respectively, its complement.

\begin{prop}[Orbits of $K$ intersecting
  $\Omega_{S,\Theta,Z}$]
\label{orbits}
Let $S \in (0,\pi/4]$, $\Theta \in (0,\pi/2]$,
and $Z \in (0,\Theta/2]$.
\begin{enumerate}[(i)]
  \item \label{entry}
  The set of entry points of
  $\Omega_{S,\Theta,Z}$ for $K$ is
  $
    (R^-_{S,\Theta,Z} \cup R^+_{S,\Theta,Z})
    \setminus
    (
      T^-_{S,\Theta,Z}
      \cup
      P_{S,\theta,Z}
      \cup
      T^+_{S,\Theta,Z}
    )
  $,

  \item \label{exit}
  the set of exit points of $\Omega_{S,\Theta,Z}$
  for $K$ is
  $
    T^-_{S,\Theta,Z}
    \cup
    P_{S,\theta,Z}
    \cup
    T^+_{S,\Theta,Z}
  $,

  \item \label{reverse_entry}
  the set of entry points of $\Omega_{S,\Theta,Z}$
  for $-K$ is
  $
    (
      T^-_{S,\Theta,Z}
      \cup
      P_{S,\theta,Z}
      \cup
      T^+_{S,\Theta,Z}
    )
    \setminus 
    (R^-_{S,\Theta,Z} \cup R^+_{S,\Theta,Z})
  $,

  \item \label{reverse_exit}
  the set of exit points of $\Omega_{S,\Theta,Z}$
  for $-K$ is
  $
    R^-_{S,\Theta,Z} \cup R^+_{S,\Theta,Z}
  $,

  \item \label{unique_exit}
  each orbit of $K$ intersects
  $
    T^-_{S,\Theta,Z}
    \cup
    P_{S,\Theta,Z}
    \cup
    T^+_{S,\Theta,Z} 
  $
  in at most one point,
  and

  \item \label{boundary}
  each orbit of $K$ through
  $
    (R^-_{S,\Theta,Z} \cup R^+_{S,\Theta,Z})
    \cap
    (
      T^-_{S,\Theta,Z}
      \cup
      P_{S,\Theta,Z}
      \cup
      T^+_{S,\Theta,Z} 
    )
  $
  intersects $\Omega_{S,\Theta,Z}$
  at a single point.

\end{enumerate}
\end{prop}

\begin{proof}
We prefer to present the proof in $\R^4$
rather than $\C^2$
(under the identification
$
  (x_1,x_2,x_3,x_4)
  \mapsto
  (x_1+ix_2,x_3+ix_4)
$),
and we choose to omit the parameters $S,\Theta,Z$
from the notation for $\Omega$ and each face.
We parametrize these faces as follows:
\begin{equation}
\begin{gathered}
  T^{\pm}
    :=
    \left\{
      \left(
        \cos \frac{Z}{2} \cos r, ~
        \pm \sin \frac{Z}{2} \sin r, ~
        \cos \theta \sin r, ~
        \sin \theta \sin r
      \right) 
    \right\}_{
      r \in [0,S],
      \;
      \theta \in
        \left[
          -\frac{\Theta}{2},\frac{\Theta}{2}
        \right],
       },
  \\
  P
    :=
    \left\{
      (
        \cos S \cos z, ~
        \cos S \sin z, ~
        \sin S \cos (z+\theta), ~
        \sin S \sin(z+\theta)
      )
    \right\}_{
        \theta \in
        \left[
          -\frac{\Theta}{2}, \frac{\Theta}{2}
        \right],
        \;
        z \in
          \left[
            -\frac{Z}{2}, \frac{Z}{2}
          \right]
    },
  \\
  R^{\pm}
    :=
    \left\{
      \left(
        \cos r \cos z, ~
        \cos r \sin z, ~
        \sin r
          \cos \left(z \pm \frac{\Theta}{2}\right), ~
        \sin r 
          \sin \left(z \pm \frac{\Theta}{2}\right)
      \right)
    \right\}_{
      r \in [0,S],
      \;
      z \in
        \left[
          -\frac{Z}{2}, \frac{Z}{2}
        \right]
    }.
\end{gathered}
\end{equation}

On each face we choose
the continuous unit normal
which agrees, on the interior of the face,
with the outward normal to $\Omega$:
\begin{equation}
\begin{gathered}
  \nu^{T^\pm}
    :=
    \left(
      -\sin \frac{Z}{2}, ~
      \pm \cos \frac{Z}{2}, ~
      0, ~
      0
    \right),
  \\
  \nu^P
    :=
    (
      -\sin S \cos z, ~
      -\sin S \sin z, ~
      \cos S \cos (z+\theta), ~
      \cos S \sin (z+\theta)
    ),
  \\
  \nu^{R^\pm}
    :=
    \left(
      \pm \sin r \sin z, ~
      \mp \sin r \cos z, ~
      \mp \cos r
        \sin \left(z \pm \frac{\Theta}{2}\right), ~
      \pm \cos r
        \cos \left(z \pm \frac{\Theta}{2}\right)
    \right).
\end{gathered}
\end{equation}

On the other hand,
(making implicit use of the above parametrizations)
\begin{equation}
\begin{gathered}
  K|_{T^\pm}
    =
    \left(
      -\cos \theta \sin r, ~
      0, ~
      \cos \frac{Z}{2} \cos r, ~
      0
    \right),
  \\
  K|_P
    =
    \left(
      -\sin S \cos (z+\theta), ~
      0, ~
      \cos S \cos z, ~
      0
    \right),
  \\
  K|_{R^\pm}
    =
    \left(
      -\sin r
        \cos \left(z \pm \frac{\Theta}{2}\right), ~
      0, ~
      \cos r \cos z, ~
      0
    \right).
\end{gathered}
\end{equation}

Recalling that
$0 \leq r \leq S \leq \pi/4$,
$\abs{\theta} \leq \Theta/2 \leq \pi/4$,
and
$\abs{z} \leq Z/2 \leq \Theta/4$,
we then compute
\begin{equation}
\label{exit_through_triangles}
  K|_{T^\pm} \cdot \nu^{T^\pm}
  =
  \sin \frac{Z}{2} \sin r \cos \theta
  \geq
  0,
  \quad
  \text{with equality iff $r=0$},
\end{equation}
\begin{equation}
\label{exit_through_parallelogram}
  K|_P \cdot \nu^P
    =
    \cos z \cos (z+\theta)
    >
    0,
\end{equation}
\begin{equation}
\begin{aligned}
  K|_{R^\pm} \cdot \nu^{R^\pm}
    &=
    -
    \left(
      \cos^2 r \sin \frac{\Theta}{2}
      \pm \sin z
          \cos \left(z \pm \frac{\Theta}{2}\right)
    \right)
    \\
    &=
    -\frac{1}{2}
    \left(
      \sin \frac{\Theta}{2} \cos 2r
      +\sin
         \left(\frac{\Theta}{2} \pm 2z\right)
    \right)
    \\
    &\leq 0,
    \quad
    \text{with equality iff 
      $r=S=\frac{\pi}{4}$
      and
      $z=\frac{\mp Z}{2}=\frac{\mp \Theta}{4}$
    }.
  \end{aligned}
\end{equation}

The above inequalities establish items
\ref{entry}, \ref{exit},
\ref{reverse_entry},
and \ref{reverse_exit}
except at the two points
\begin{equation}
  T^\pm \cap R^+ \cap R^-
  =
  (\pm Z/2,0,0,0)
\end{equation}
and, in the case that
$S=\frac{\pi}{4}$
and
$Z=\frac{\Theta}{2}$,
at the two additional points
\begin{equation}
\label{RTP}
  R^\pm \cap T^\mp \cap P
  =
  \frac{1}{\sqrt{2}}
  \left(
    \cos \frac{\Theta}{4}, ~
    \mp \sin \frac{\Theta}{4}, ~
    \cos \frac{\Theta}{4}, ~
    \pm \sin \frac{\Theta}{4}
  \right).
\end{equation}
The orbit of $T^\pm \cap R^+ \cap R^-$
is, according to the above calculation,
tangential there to $\Sigma_{\pm Z/2}$
but is easily seen
(since $0<z<\pi/2$)
to have, at the same point, vector-valued curvature
directed along $\acirc{1}$ away from $(1,0,0,0)$,
while $\Sigma_{\pm Z/2}$ is totally geodesic,
confirming that $T^\pm \cap R^+ \cap R^-$
is an exit point for both $K$ and $-K$.
In fact this completes the proof
of \ref{entry} and \ref{exit}.

In the case that $S=\frac{\pi}{4}$
and $Z=\frac{\Theta}{2}$
the orbit through
$R^\pm \cap T^{\mp} \cap P$
is parametrized by
\begin{equation}
  \alpha_\pm(t)
  =
  \frac{1}{\sqrt{2}}
  \left(
    (\cos t-\sin t) \cos \frac{\Theta}{4}, ~
    \mp \sin \frac{\Theta}{4}, ~
    (\cos t + \sin t)\cos \frac{\Theta}{4}, ~
    \pm \sin \frac{\Theta}{4} 
  \right),
\end{equation}
while
$\T^{\pm \Theta/2}$ is the intersection with $\Sph^3$
of the null set of
\begin{equation*}
  f_\pm(x_1,x_2,x_3,x_4)
  :=
  (x_2x_3 - x_1 x_4) \cos \frac{\Theta}{2}
    \pm (x_1x_3 + x_2x_4) \sin \frac{\Theta}{2};
\end{equation*}
then
\begin{equation}
  f_\pm(\alpha_\pm(t))
  =
  \pm \frac{1}{2}\left(\sin \frac{\Theta}{2}\right)
  \left(
    \cos^2 \frac{\Theta}{4} \cos 2t
    -\cos \frac{\Theta}{2} \cos t
    -\sin^2 \frac{\Theta}{4}
  \right),
\end{equation}
and so
\begin{equation}
  \mp 2\left.\frac{d^2}{dt^2}\right|_{t=0}f_\pm(\alpha_\pm(t))
  =
  \left(\sin \frac{\Theta}{2}\right)
    \left(
      1+2\cos^2 \frac{\Theta}{4}  
    \right)
  >
  0,
\end{equation}
but
$
 \pm f_\pm|_\Omega
 \geq 0
$,
completing the proof of claims
\ref{reverse_entry} and \ref{reverse_exit}.

We turn to claim \ref{unique_exit}.
Using the fact that the orbit of
$(x_1,x_2,x_3,x_4)$ can be parametrized by
\begin{equation}
\label{orbit}
  \beta(t)
  :=
    (
      x_1 \cos t - x_3 \sin t, ~
      x_2, ~
      x_1 \sin t + x_3 \cos t, ~
      x_4
    ),
\end{equation}
we observe that if
$(x_1,x_2,x_3,x_4) \in \Sigma^{\pm Z/2}$,
then
$x_2 = \pm x_1 \tan \frac{Z}{2}$
and $\beta(t) \in \Sigma^{\pm Z/2}$
when
\begin{equation}
  x_1 \cos t - x_3 \sin t
  =
  x_1,
\end{equation}
which has two solutions (modulo $2\pi\Z$)
unless $x_3=0$, in which case only one solution.
The computation \ref{exit_through_triangles}
and the analysis in the paragraph
immediately following \ref{RTP}
reveal that
at $t=0$
the orbit $\beta(t)$
not only exits $\Omega$
but also enters the component
$\Sph^3 \setminus \Sigma^{\pm Z/2}$
disjoint from $\Omega$.
Thus the least strictly positive $t$,
at which
$\beta(t) \in T^+ \cup P \cup T^-$,
which we will call $t'$,
is at least the least positive $t$
satisfying \ref{orbit}.
Since $x_1,x_3 \geq 0$
(simply because $(x_1,x_2,x_3,x_4) \in \Omega$),
this last value is strictly greater than $\pi$.
On the other hand,
the order-$4$ subgroup
$
  \left\langle
    \rot_{C^{\pi/2}_{\pi/2}}^{\pi/2}
  \right\rangle
$
acts simply transitively on the components
of
$\Sph^3 \setminus (\Sigma_{\pi/2} \cup \Sigma^{\pi/2})$,
a single one of which contains the interior of $\Omega$.
It follows that $t' \geq 3\pi/2$.

Now
suppose that $(x_1,x_2,x_3,x_4) \in P$.
Then
$x_1^2+x_2^2 = \cos^2 S$,
and $\beta(t) \in T^S$ when
\begin{equation}
  (\sin t)
  \left[
    (x_3^2-x_1^2) \sin t
    - 2x_1x_3 \cos t
  \right]
  =
  0,
\end{equation}
which
(under the assumption $(x_1,x_2,x_3,x_4) \in P$)
has four solutions (modulo $2\pi\Z$).
Since
both
$\refl_{\Sigma_{\pi/2}}$
and $\refl_{\Sigma^{\pi/2}}$
are symmetries of
both $T^S$ and $C^0_0$,
each component of
$\Sph^3 \setminus (\Sigma_{\pi/2} \cup \Sigma^{\pi/2})$
contains a solution.
On the other hand,
\ref{exit_through_parallelogram}
shows that at $t=0$
the orbit $\beta$ exits
the component of $\Sph^3 \setminus T^S$
containing $\Omega$.
It follows that the least strictly positive $t$
for which
$\beta(t) \in T^- \cup P \cup T^+$
is at least $3\pi/2$.
Since each nontrivial orbit of $K$
has period $2\pi$,
the conclusions of this and the preceding paragraph
show that intersection of any orbit of $K$
with $T^- \cup P \cup T^+$
has cardinality at most one,
confirming \ref{unique_exit}.

Finally, let $O$ be an orbit through a point
$p \in (R^- \cup R^+) \cap (T^- \cup P \cup T^+)$.
In particular, by \ref{exit},
$p$ is an exit point for $K$.
By \ref{unique_exit},
if $O \cap \Omega \neq \{p\}$,
then there is also an entry point for $K$
on $O$,
and then $p$ would also have to an entry point
for $-K$,
violating \ref{reverse_exit}
and so confirming \ref{boundary}.
\end{proof}

Now we can prove uniqueness
and $K$ graphicality
(as defined in the statement below)
of the solution
to a certain Plateau problem in $\Omega_{S,\Theta,Z}$.

\begin{thm}[Uniqueness and graphicality 
  for the Plateau problem]
\label{unique_disc}
Let $S \in (0,\pi/4]$, $\Theta \in (0,\pi/2]$,
and $Z \in (0,\Theta/2]$.
Then there is a unique connected minimal surface
(with boundary) $D$,
an embedded disc,
such that $D \subset \Omega_{S,\Theta,Z}$
and
\begin{equation*}
\begin{aligned}
  \partial D
  =
  (R^-_{S,\Theta,Z} \cup R^+_{S,\Theta,Z})
    \cap
    (
      T^-_{S,\Theta,Z}
      \cup
      P_{S,\Theta,Z}
      \cup
      T^+_{S,\Theta,Z} 
    ),
  \end{aligned}
\end{equation*}
the hexagon which is the union
of the edges of $\Omega_{S,\Theta,Z}$
except
$R_{S,\Theta,Z}^- \cap R_{S,\Theta,Z}^+$,
$P_{S,\Theta,Z} \cap T_{S,\Theta,Z}^+$,
and $P_{S,\Theta,Z} \cap T_{S,\Theta,Z}^-$,
Furthermore, $D$ is graphical with respect to $K$
in the sense that no orbit of $K$ intersects $D$
in more than one point.
\end{thm}

\begin{proof}
By \cite{MeeksYauEmbedded}*{Theorem 1}
there exists an embedded minimal disc
$D \subset \Omega_{S,\Theta,Z}$
with the prescribed boundary.
Now suppose $M$ is a connected minimal surface
in $\Omega_{S,\Theta,Z}$
with $\partial M = \partial D$.
We will show, by contradiction,
that no orbit of $K$ intersects
$D$ and $M$ at distinct points.
To this end suppose then that such an orbit does exist.
This assumption,
together with the fact that
$\Omega_{S,\Theta,Z}$
lies in the closure of a single component
of $\Sph^3 \setminus (\Sigma_{\pi/2} \cup \Sigma^{\pi/2})$,
implies the existence
of
$t' \in \R \setminus 2\pi\Z$ and $\epsilon>0$
such that
the intersection $D \cap \rot_{C^{\pi/2}_{\pi/2}}^t M$
is nonempty for $t=t'$
but empty either for $t \in (t',t'+\epsilon)$
or $t \in (t'-\epsilon, t')$.
Item \ref{boundary} of \ref{orbits}
ensures that the intersection, when $t=t'$,
cannot include any points in the boundary
of either surface ($D$ or the rotated $M$).
The maximum principle then forces
the surfaces to coincide,
but since $t' \not\in 2\pi\Z$, their boundaries,
again by \ref{orbits}.\ref{boundary},
are actually disjoint,
which contradiction establishes our claim
that no orbit of $K$
intersects $D$ and $M$ at distinct points.

The graphicality claim is simply
the special case of this last fact when $M=D$.
To prove the uniqueness,
note that it follows from the embeddedness of $D$,
the fact that $\Omega_{S,\Theta,Z}$ is a $3$-ball,
and items \ref{entry} and \ref{exit}
of \ref{orbits}
that each orbit of $K$ intersecting $\Omega_{S,\Theta,Z}$
also intersects $D$,
and so, if $M \neq D$,
there would exist an orbit of $K$
intersecting $M$ and $D$ at distinct points,
which contradiction completes the proof.
\end{proof}

\subsection*{Proof of Theorem \ref{odd_cs_uniqueness}}
Now Theorem \ref{odd_cs_uniqueness}
is obtained as a corollary
of Theorem \ref{unique_disc}
as follows.
Choe and Soret construct $T^o_{k,2km}$
(in the proof of \cite{ChoeSoretTordesing}*{Theorem 1})
by (i) solving the Plateau problem,
for an embedded disc---called $H$
in \cite{ChoeSoretTordesing}---in a
prism---$U_1^1$ in the notation
of \cite{ChoeSoretTordesing}---congruent to
$\Omega_{\pi/4,\pi/k,\pi/2km}$
with prescribed boundary a geodesic hexagon
as in Theorem \ref{unique_disc} (shifted by the congruence)
and then (ii) checking that $H$
can be extended to a closed embedded surface
by applying the group---$G^o$ in
\cite{ChoeSoretTordesing}---generated
by the reflections through the great circles
containing the sides of the hexagon.
Theorem \ref{unique_disc}
establishes that
the solution to the preceding Plateau problem
is unique,
and so the surface $T^o_{k,2km}$
is uniquely determined by its construction.

Similarly,
from the reflection principle
it is clear that any closed embedded minimal surface
whose intersection with $U_1^1$
is nonempty and has boundary $\partial H$
is uniquely determined by its intersection with $U_1^1$,
but Theorem \ref{unique_disc}
implies that this intersection is $H$,
and so the surface is $T^o_{k,2km}$.

Finally,
to see that $T^o_{k,2km}$ agrees
with the minimal surface constructed
in \cite{KWtordesing}*{Theorem 7.1}
from the initial surface
$M(k,2m,1,1,0)$
(\cite{KWtordesing}*{Definition 4.13})
under the assumption that $m$
is large enough to ensure the existence
of the latter surface,
it suffices to confirm that,
modulo congruence,
the intersection of the latter surface
with $U_1^1$
is nontrivial and has boundary $\partial H$.
For this of course it is necessary
to refer to some of the details of \cite{KWtordesing},
but we explain here the key points to check.

In \cite{KWtordesing}
the initial surface $M(k,2m,1,1,0)$
is constructed explicitly
and then graphically perturbed
to minimality
by solving the corresponding
nonlinear partial differential equation,
first by constructing a suitable right-inverse
to the linearized operator
and then by appealing to the Schauder fixed point theorem
to solve the nonlinear problem.
The explicit construction
(see \cite{KWtordesing}*{4.14}
and the supporting definitions)
of the initial surface $M(k,2m,1,1,0)$
makes it clear that $M(k,2m,1,1,0)$
(modulo congruence)
not only has the desired properties that
its intersection with $U_1^1$
is nontrivial and has boundary $\partial H$
but moreover is invariant under $G^o$.

In \cite{KWtordesing}
we enforce a proper subgroup of $G^o$,
but it is easy to see,
given the $G^o$ invariance of this initial surface,
that in this case
the larger group $G^o$ could be enforced.
(In fact the right-inverse
to the linearized problem constructed in
\cite{KWtordesing}*{Section 6}
respects $G^o$, without modification.)
Although in \cite{KWtordesing}
we did not prove uniqueness
of the minimal perturbations
of the initial surfaces considered there,
one could in this case
simply appeal to the inverse function theorem
in place of the Schauder fixed point theorem,
and thereby establish uniqueness,
assuming $m$ sufficiently large.
Thus we conclude that the minimal surface
obtained from $M(k,2m,1,1,0)$
is invariant under $G^o$,
but since it is also a small perturbation
of a surface whose intersection
with $U_1^1$ is nonempty
and has boundary $H$,
its own intersection with $U_1^1$
will have these same properties,
completing the proof.
\qed

\section{Desingularizations of three Clifford tori
intersecting pairwise orthogonally}
\label{sec:three_tori}

\subsection*{Initial configuration}
Fix a Clifford torus $\T_2$ containing $\acirc{1}$.
There is a unique Clifford torus
$
  \T_3
  =
  \rot_{\acirc{1}}^{\pi/2}\T_2
  =
  \rot_{\acirc{1}}^{-\pi/2}\T_2
$
intersecting $\T_2$ orthogonally along $\acirc{1}$.
Set $\T_1:=\{d^{\Sph^3}_{\acirc{1}}=\pi/4\}$,
another Clifford torus.
More concretely, we will take
\begin{equation}
  \T_2
    :=
    \{(e^{\ii s} \cos r, e^{\ii s} \sin r)\}_{r,s \in \R}
    =
    \phopf^{-1}(\{y=0\}),
\end{equation}
and so
\begin{equation}
\begin{gathered}
  \T_1
    =
    \{2^{-1/2}(e^{\ii s}, e^{\ii t})\}_{s,t \in \R}
    =
    \phopf^{-1}(\{z=0\}),
  \\
  \T_3
    =
    \{e^{\ii s}(\cos r, \ii \sin r)\}_{r,s \in \R}
    =
    \phopf^{-1}(\{x=0\}).
\end{gathered}
\end{equation}

The following claims are immediate consequences
of the foregoing definitions.

\begin{prop}[Collection $\acol$]
\hspace{.1in}

\begin{enumerate}[(i)]
\item The set $\bigcup_{i \neq j} (\T_i \cap \T_j)$
consists of six pairwise disjoint great circles,
the collection of which we call $\acol$.

\item Each of the three $\T_i$
coincides with $\{d^{\Sph^3}_C=\pi/4\}=\{d^{\Sph^3}_{C^\perp}=\pi/4\}$
for precisely two choices of $C \in \acol$,
along each of which the other two tori
intersect orthogonally.

\item For any $C,C' \in \acol$
the function $d^{\Sph^3}_C$ is constant on $C'$. 

\item Each $C \in \acol$ has four nearest neighbors
(excluding $C$ itself) in $\acol$,
each at distance $\pi/4$,
namely
$(\bigcup \acol) \setminus (C \cup C^\perp)$.
\end{enumerate}
\end{prop}

Namely, we have
\begin{equation}
\begin{gathered}
  \acol
    =
    \{\acirc{i}\}_{i=1}^3 \cup \{\acirc{i}^\perp\}_{i=1}^3,
  \\
  \acirc{2}
    :=
    \left\{
      \frac{e^{\ii s}}{\sqrt{2}}
      (1,\ii)
    \right\}_{s \in \R}
    =
    \phopf^{-1}(\{(0,-\tfrac{1}{2},0)\}),
  \\
  \acirc{2}^\perp
    =
    \left\{
      \frac{e^{\ii s}}{\sqrt{2}}
      (1,-\ii)
    \right\}_{s \in \R}
    =
    \phopf^{-1}(\{(0,\tfrac{1}{2},0)\}),
  \\
  \acirc{3}
    :=
    \left\{
      \frac{e^{\ii s}}{\sqrt{2}}
      (1,1)
    \right\}_{s \in \R}
    =
    \phopf^{-1}(\{(\tfrac{1}{2},0,0)\}),
  \\
  \acirc{3}^\perp
    =
    \left\{
      \frac{e^{\ii s}}{\sqrt{2}}
      (1,-1)
    \right\}_{s \in \R}
    =
    \phopf^{-1}(\{(-\tfrac{1}{2},0,0)\}).
\end{gathered}
\end{equation}

\subsection*{Packing}
For each $C \in \acol$ define the closed solid torus
\begin{equation}
  T_C := \{d^{\Sph^3}_C \leq \pi/8\}
    \quad
    (C \in \acol).
\end{equation}

\begin{prop}[Collection $\rcol$]
\hspace{.1in}

\begin{enumerate}[(i)]
\item For any $C \neq C' \in \acol$
the intersection $T_C \cap T_{C'}$
is empty if $C'=C^\perp$
and otherwise consists
of a single great circle
lying on $\partial T_C \cap \partial T_{C'}$.
In particular $\{T_C \; : \; C \in \acol\}$
consists of six solid tori,
whose interiors are pairwise disjoint.

\item We have the equality
\begin{equation*}
  \bigcup_{C \neq C' \in \acol} (T_C \cap T_{C'})
  =
  \bigcup_{i=1}^3 \bigcup_{C \in \acol}
    (\T_i \cap \partial T_C),
\end{equation*}
and this set consists of twelve great circles,
the collection of which we call $\rcol$.
\end{enumerate}
\end{prop}

\begin{proof}
Suppose $C \neq C' \in \acol$.
and $p \in T_C \cap T_{C'}$.
Then the distance between $C$ and $C'$
is at most $\pi/4$,
so in fact $C'$
must be one of the four nearest neighbors
of $C$ in $\acol$
and therefore $d^{\Sph^3}_C|_{C'}=\pi/4$.
Moreover, $p$ must lie on a great circle
intersecting $C$ and $C'$ orthogonally,
and so $p$ lies on the unique torus
$\T \in \{\T_1,\T_2,\T_3\}$
containing $C \cup C'$.
We have
$
  \T \cap \{d^{\Sph^3}_C=\pi/8\}
  =
  \T \cap \{d^{\Sph^3}_{C'}=\pi/8\}
$,
and this set is a great circle.
Finally, each $\T_i$ is disjoint
from two of the six tori
in $\{T_{C''}\}_{C'' \in \acol}$
and intersects the boundaries of the remaining four
in a total of four great circles,
while the $\T_i$ themselves intersect
(in pairs) only along the circles in $\acol$. 
\end{proof}

\begin{figure}
  \includegraphics[width=3in]{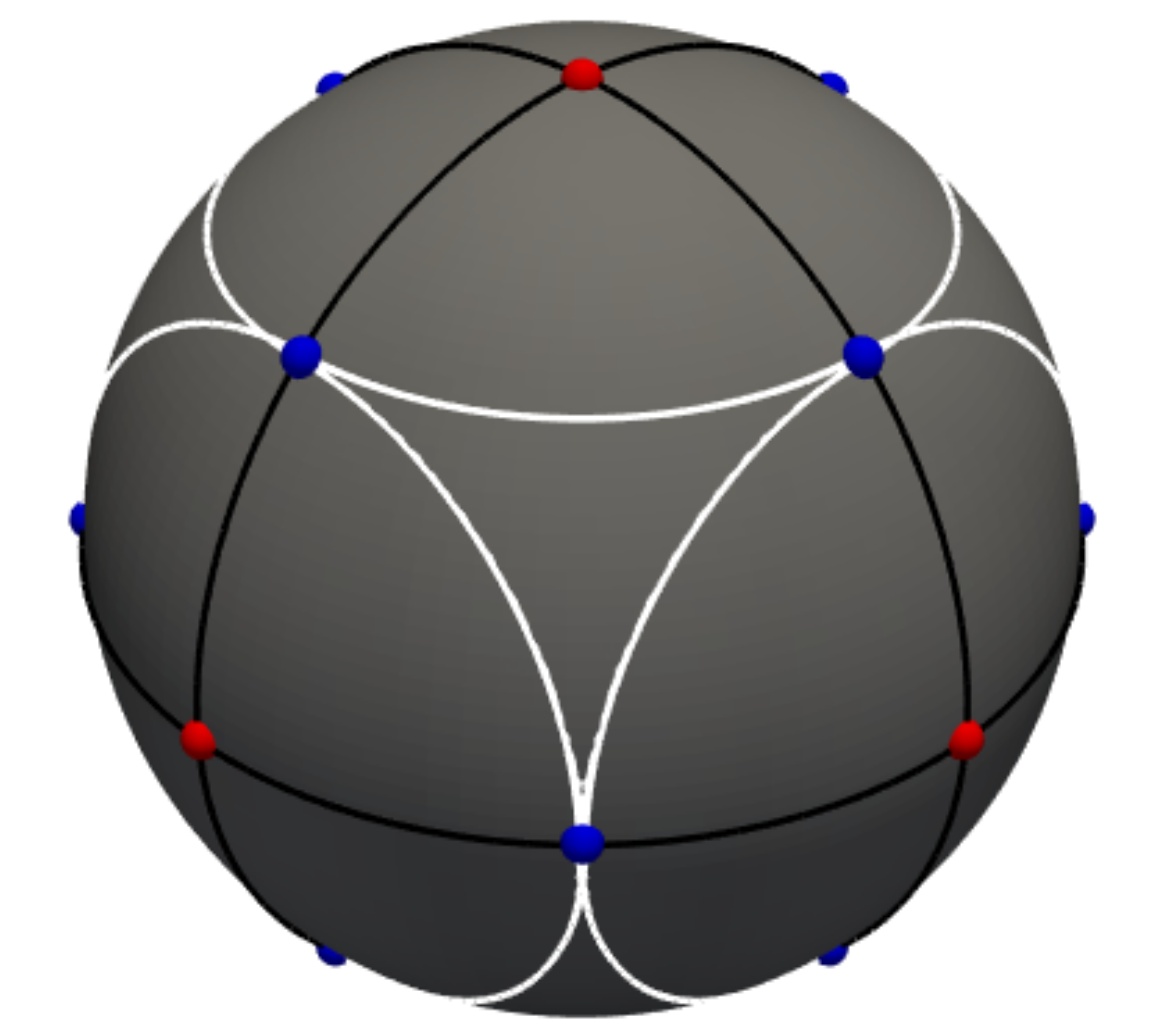}
  \caption*{The images under $\phopf$
  of $\{\T_i\}_{i=1}^3$ (in black),
  $\acol$ (red),
  $\{\partial T_C\}_{C \in \acol}$ (white),
  and $\rcol$ (blue).}
\end{figure}

Explicitly,
\begin{equation}
\label{rcol}
\begin{gathered}
  \rcol
    =
    \{
      \Gamma_{12},
      \Gamma_{13},
      \Gamma_{12^\perp},
      \Gamma_{13^\perp},
      \Gamma_{1^\perp2},
      \Gamma_{1^\perp3},
      \Gamma_{1^\perp2^\perp},
      \Gamma_{1^\perp3\perp},
      \Gamma_{23},
      \Gamma_{23^\perp},
      \Gamma_{2^\perp3},
      \Gamma_{2^\perp3^\perp}
    \},
  \\
  \Gamma_{12}
    :=
    \left\{
      e^{\ii s}
      \left(
        \cos \frac{\pi}{8},
        \ii \sin \frac{\pi}{8}
      \right)
    \right\}_{s \in \R}
    =
    \phopf^{-1}
      \left(
        \left\{
          \frac{\sqrt{2}}{4}(0,-1,1)
        \right\}
      \right),
  \\
  \Gamma_{1^\perp2^\perp}
    :=
    \left\{
      e^{\ii s}
      \left(
        \sin \frac{\pi}{8},
        -\ii \cos \frac{\pi}{8}
      \right)
    \right\}_{s \in \R}
    =
    \phopf^{-1}
      \left(
        \left\{
          \frac{\sqrt{2}}{4}(0,1,-1)
        \right\}
      \right)
    =
    \Gamma_{12}^\perp,
  \\
  \Gamma_{13}
    :=
    \left\{
      e^{\ii s}
      \left(
        \cos \frac{\pi}{8},
        \sin \frac{\pi}{8}
      \right)
    \right\}_{s \in \R}
    =
    \phopf^{-1}
      \left(
        \left\{
          \frac{\sqrt{2}}{4}(1,0,1)
        \right\}
      \right),
  \\
  \Gamma_{1^\perp3^\perp}
    :=
    \left\{
      e^{\ii s}
      \left(
        \sin \frac{\pi}{8},
        -\cos \frac{\pi}{8}
      \right)
    \right\}_{s \in \R}
    =
    \phopf^{-1}
      \left(
        \left\{
          \frac{\sqrt{2}}{4}(-1,0,-1)
        \right\}
      \right)
    =
    \Gamma_{13}^\perp,
    \\
  \Gamma_{12^\perp}
    :=
    \left\{
      e^{\ii s}
      \left(
        \cos \frac{\pi}{8},
        -\ii \sin \frac{\pi}{8}
      \right)
    \right\}_{s \in \R}
    =
    \phopf^{-1}
      \left(
        \left\{
          \frac{\sqrt{2}}{4}(0,1,1)
        \right\}
      \right),
  \\
  \Gamma_{1^\perp2}
    :=
    \left\{
      e^{\ii s}
      \left(
        \sin \frac{\pi}{8},
        \ii \cos \frac{\pi}{8}
      \right)
    \right\}_{s \in \R}
    =
    \phopf^{-1}
      \left(
        \left\{
          \frac{\sqrt{2}}{4}(0,-1,-1)
        \right\}
      \right)
    =
    \Gamma_{12^\perp}^\perp,
    \\
  \Gamma_{13^\perp}
    :=
    \left\{
      e^{\ii s}
      \left(
        \cos \frac{\pi}{8},
        -\sin \frac{\pi}{8}
      \right)
    \right\}_{s \in \R}
    =
    \phopf^{-1}
      \left(
        \left\{
          \frac{\sqrt{2}}{4}(-1,0,1)
        \right\}
      \right),
  \\
  \Gamma_{1^\perp3}
    :=
    \left\{
      e^{\ii s}
      \left(
        \sin \frac{\pi}{8},
        \cos \frac{\pi}{8}
      \right)
    \right\}_{s \in \R}
    =
    \phopf^{-1}
      \left(
        \left\{
          \frac{\sqrt{2}}{4}(1,0,-1)
        \right\}
      \right)
    =
    \Gamma_{13^\perp}^\perp,
  \\
  \Gamma_{23}
    :=
    \left\{
      \frac{e^{\ii s}}{\sqrt{2}}
      (1,e^{\ii \pi/4})
    \right\}_{s \in \R}
    =
    \phopf^{-1}
      \left(
        \left\{
          \frac{\sqrt{2}}{4}(1,-1,0)
        \right\}
      \right),
  \\
  \Gamma_{2^\perp3^\perp}
    :=
    \left\{
      \frac{e^{\ii s}}{\sqrt{2}}
      (1,e^{5\ii\pi/4})
    \right\}_{s \in \R}
    =
    \phopf^{-1}
      \left(
        \left\{
          \frac{\sqrt{2}}{4}(-1,1,0)
        \right\}
      \right)
    =
    \Gamma_{23}^\perp,
  \\
  \Gamma_{2^\perp3}
    :=
    \left\{
      \frac{e^{\ii s}}{\sqrt{2}}
      (1,e^{-\ii\pi/4})
    \right\}_{s \in \R}
    =
    \phopf^{-1}
      \left(
        \left\{
          \frac{\sqrt{2}}{4}(1,1,0)
        \right\}
      \right),
  \\
  \Gamma_{23^\perp}
    :=
    \left\{
      \frac{e^{\ii s}}{\sqrt{2}}
      (1,e^{3\ii\pi/4})
    \right\}_{s \in \R}
    =
    \phopf^{-1}
      \left(
        \left\{
          \frac{\sqrt{2}}{4}(-1,-1,0)
        \right\}
      \right)
    =
    \Gamma_{2^\perp3}^\perp.
\end{gathered}
\end{equation}

\subsection*{Scaffolding}
Now let $N \geq 1$ be a given integer.
We fix a set $\ptcol^N_{\acirc{1}}$
of $2N$ equally spaced points on $\acirc{1}$.
(In particular
$\ptcol^N_{\acirc{1}}$ is invariant
under antipodal reflection.)
We take
\begin{equation}
  \ptcol^N_{\acirc{1}}
  :=
  \{(e^{j\pi \ii/N},0)\}_{j=1}^{2N}.
\end{equation}
Then there are unique collections
$\scaf^N_{\T_2}$ and $\scaf^N_{\T_3}$
of great circles
on $\T_2$ and $\T_3$ respectively
such that
$
  \bigcup_{L \in \scaf^N_{\T_i}} (L \cap \acirc{1})
  =
  \ptcol^N_{\acirc{1}}
$
for each $i \in \{2,3\}$.
Each of these collections contains $N$ great circles,
and each of these great circles
intersects $\acirc{1}$ orthogonally
and also $\acirc{1}^\perp$ orthogonally at two points.
We have
\begin{equation}
  \scaf^N_{\T_2}
    =
    \{
      \{
        e^{j\pi \ii/N}(\cos r, \sin r)
      \}_{r \in \R}
    \}_{j=1}^N,
  \quad
  \scaf^N_{\T_3}
    =
    \{
      \{
        e^{j\pi \ii/N}(\cos r, \ii\sin r)
      \}_{r \in \R}
    \}_{j=1}^N.
\end{equation}

\begin{prop}[Collection $\ptcol^N_{\acirc{1}^\perp}$]
\label{pts-on-C1perp}
If $N$ is even, then
$
  \bigcup_{L \in \scaf^N_{\T_2}} (L \cap \acirc{1}^\perp)
  =
  \bigcup_{L \in \scaf^N_{\T_3}} (L \cap \acirc{1}^\perp)
$
and consists of $2N$ equally spaced points,
which we call $\ptcol^N_{\acirc{1}^\perp}$.
\end{prop}

\begin{proof}
Two circles meeting one another and $\acirc{1}$
orthogonally at two points
meet $\acirc{1}^\perp$ orthogonally
and (collectively) at four equally spaced points.
\end{proof}

We henceforth assume $N$ even.
We have
\begin{equation}
  \ptcol^N_{\acirc{1}^\perp}
    =
    \{(0,e^{j\pi \ii/N})\}_{j=1}^{2N}.
\end{equation}

For each $i \in \{2,3\}$
each $L \in \scaf^N_{\T_i}$
also orthogonally intersects
each of the two circles in $\acol$
along which $\T_i$ and $\T_1$ intersect.
For each $C \in \acol$ we set
\begin{equation}
  \ptcol^N_C
  :=
  \bigcup_{i=2}^3
    \bigcup_{L \in \scaf^N_{\T_i}}
      (L \cap C),
\end{equation}
a collection of $2N$ equally spaced points,
consistent with the pre-existing definitions
in case $C \in \{\acirc{1},\acirc{1}^\perp\}$.
We further define
\begin{equation}
  \ptcol^N := \bigcup_{C \in \acol} \ptcol^N_C.
\end{equation}
Then
\begin{equation}
\begin{gathered}
  \ptcol^N_{\acirc{2}}
    =
    \{
      e^{j \pi \ii / N}
        (1,\ii) / \sqrt{2}
    \}_{j=1}^{2N},
  \quad
  \ptcol^N_{\acirc{2}^\perp}
    =
    \{
      e^{j \pi \ii / N}
        (1,-\ii) / \sqrt{2}
    \}_{j=1}^{2N},
  \\
  \ptcol^N_{\acirc{3}}
    =
    \{
      e^{j \pi \ii / N}
        (1,1) / \sqrt{2}
    \}_{j=1}^{2N},
  \quad
  \ptcol^N_{\acirc{3}^\perp}
    =
    \{
      e^{j \pi \ii / N}
        (1,-1) / \sqrt{2}
    \}_{j=1}^{2N}.
\end{gathered}
\end{equation}

\begin{prop}[Collection $\scaf^N_{\T_1}$]
Let $N$ be divisible by $4$.
Then for any of the four $C \in \acol$ on $\T_1$
the $N$ great circles on $\T_1$
orthogonally intersecting $C$ (collectively)
on $\ptcol^N_C$
intersect (collectively)
each circle $C' \in \acol$ on $\T_1$
in precisely $\ptcol^N_{C'}$.
Thus the collection of these $N$ great circles,
which we call $\scaf^N_{\T_1}$,
is independent of $C$.
\end{prop}

\begin{proof}
Suppose $C \in \acol$ and $C \subset \T_1$
and suppose $L \subset \T_1$ is a great circle
intersecting $C$ orthogonally
at $\pm p \in \ptcol^N_C$.
Since $N$ is even,
$L \cap C^\perp \subset \ptcol^N_{C^\perp}$
(as in \ref{pts-on-C1perp}),
so suppose $C'$
is one of the two circles in $\acol$
on $\T_1$ and at distance $\pi/4$ from $C$.
We may assume $C'=\rot_{\acirc{1}}^{-\pi/2}C$.
Then $L$ intersects $C'$ at
$p' = \rot_{\acirc{1}}^{-\pi/4}\rot_{\acirc{1}^\perp}^{\pi/4}p$
and $-p'$.
Since $p \in \ptcol^N_C$ and $N \bmod 4 = 0$,
we also have
$q:=\rot_{\acirc{1}}^{\pi/4}\rot_{\acirc{1}^\perp}^{\pi/4}p \in C$,
but then by definition of $\ptcol^N_C$ and $\ptcol^N_{C'}$
the point $p'=\rot_{\acirc{1}}^{-\pi/2}q$
indeed belongs to $\ptcol^N_{C'}$.
\end{proof}

We henceforth assume $N$ divisible by $4$.
Set
\begin{equation}
\label{scaf}
  \scaf^N := \bigcup_{i=1}^3 \scaf^N_{\T_i},
\end{equation}
a collection of $3N$ great circles.
We have
\begin{equation}
  \bigcup_{L \neq L' \in \scaf^N}
    (L \cap L')
  =
  \bigcup_{L \in \scaf^N}
    \bigcup_{C \in \acol} 
      L \cap C
  =
  \ptcol^N.
\end{equation}
Concretely,
\begin{equation}
  \scaf^N_{\T_1}
  =
  \left\{
    \left\{
      \frac{e^{j\pi \ii /N}}{\sqrt{2}}
      (e^{\ii s}, e^{-\ii s})
    \right\}_{s \in \R}
  \right\}_{j=1}^N.
\end{equation}

\subsection*{Tessellations}
For each $C \in \acol$
let $\scol^N_C$
be the collection
of the $2N$ connected components
of the intersections with the solid torus $T_C$
of the $N$ great spheres
(collectively) intersecting $C$ orthogonally
on the set $\ptcol^N_C$;
each element of $\scol^N$
is thus a closed spherical cap of radius $\pi/8$.
Then
$T_C \setminus \bigcup \scol^N_C$
has $2N$ congruent components,
and
$T_C \setminus (\bigcup \scol^N_C \cup \bigcup_{i=1}^3 \T_i)$
has $8N$ congruent components,
the collection of whose closures
we call $\prcol^N_C$.
We also define
\begin{equation}
  \scol^N := \bigcup_{C \in \acol} \scol^N_C,
  \quad
  \prcol^N := \bigcup_{C \in \acol} \prcol^N_C.
\end{equation}
The below description
of the $48N$ congruent members of $\prcol^N$
follows immediately from the preceding definitions.

\begin{prop}[Geometry of $\Omega \in \prcol^N_C$]
\label{prisms_for_three_tori}
Assume $N \in 4\Z$ positive.
Let $C \in \acol$ and $\Omega \in \prcol^N_C$.
Then
\begin{enumerate}[(i)]
\item \label{congruence}
$\Omega$ is congruent to
$\Omega_{\frac{\pi}{8},\frac{\pi}{2},\frac{\pi}{N}}$
(as defined in \ref{prism_def});

\item the two rectangular faces
$R_\Omega^\pm$
lie on two (minimal) tori in $\{\T_i\}_{i=1}^3$,
the two triangular faces
$T_\Omega^\pm$
lie on two members of $\scol^N_C$,
and the parallelogram face
$P_\Omega$
lies on the torus
(of nonzero constant mean curvature)
$\partial T_C$;

\item each of the four edges
      $R_\Omega^{{\pm}_1} \cap T_\Omega^{{\pm}_2}$ 
      is an arc of length $\pi/8$
      on a great circle in $\scaf^N$;

\item each of the two edges
      $P_\Omega \cap R_\Omega^\pm$
      has constant dihedral angle $\pi/2$
      and is an arc of length $\pi/N$
      on a great circle in $\rcol$;

\item \label{four_prisms}
      for any given edge of $\Omega$
      on a circle in $\scaf^N \cup \rcol$
      there are precisely four elements
      of $\prcol^N$ 
      (including $\Omega$ itself)
      sharing that edge; and

\item \label{axial_sym} the symmetry group of $\Omega$
is the order-$2$ group
\begin{equation*}
  \{
    \mathsf{R} \in O(4)
    \; : \;
    \mathsf{R}\Omega=\Omega
  \}
  =
  \langle \refl_{L_\Omega} \rangle,
\end{equation*}
where $L_\Omega$,
the axis of $\Omega$,
is the great circle
intersecting the midpoint of $C$
and the centroid of $P_\Omega$.
\end{enumerate}
\end{prop}

For each $C \in \acol$
and for each $\Omega \in \prcol^N_C$
we define the center
of $\Omega$
to be the midpoint of $L_\Omega \cap \Omega$.
Then the collection of centers
of the $\Omega \in \prcol^N_{\acirc{1}}$ is
\begin{equation}
\label{centers}
\begin{gathered}
  \ctrcol^N_{\acirc{1}}
    :=
    \{\omega^N_{j\ell}\}_{j,\ell \in \Z},
  \\
  \omega^N_{j\ell}
    :=
      e^{\ii\pi \frac{2j+1}{2N}}
        \left(
           \cos \frac{\pi}{16}, ~
           e^{\ii\pi\frac{2\ell+1}{4}} \sin \frac{\pi}{16}
        \right),
\end{gathered}
\end{equation}
and, for any given $j,\ell \in \Z$,
we write $\Omega^N_{j\ell}$ for the unique
element of $\prcol^N$ with center $\omega^N_{j\ell}$.

\subsection*{Symmetries}
We continue to assume $N \in 4\Z$ positive,
and we define the group
\begin{equation}
\label{group_definition}
  \grp^N
    :=
    \langle
      \refl_\gamma
      \; : \;
      \gamma \in \scaf^N \cup \rcol
    \rangle.
\end{equation}

\begin{prop}[Action of $\grp^N$ on $\prcol^N$]
\label{action_on_collections}
Let $N \in 4\Z$ be positive.
Then $\grp^N$ preserves each of
$\acol$, $\{\T_i\}_{i=1}^3$, $\{T_C\}_{C \in \acol}$,
$\scaf^N$, $\ptcol^N$, $\scol^N$,
and
$\prcol^N$.
Moreover, $\prcol^N$ has at most two orbits under $\grp^N$.
\end{prop}

\begin{proof}
Granting that $\grp^N$ preserves $\prcol^N$,
the final claim is clear from the evident facts
that
$
  \langle
    \refl_L
    \; : \;
    L \in \bigcup_{i=2}^3 \scaf^N_{\T_i}
  \rangle
$
acts transitively on $\{\omega^N_{j\ell}\}_{j+\ell \in 2\Z}$
and that
$
  \langle
    \refl_\Gamma
    \; : \;
    \Gamma \in \rcol
  \rangle
$
acts transitively on $\{T_C\}_{C \in \acol}$.
To check the remaining claims
it suffices to do so at the level of the generators.

First suppose $L \in \scaf^N$.
Then $L \subset \T_i$ for some $i \in \{1,2,3\}$,
and $\refl_L$ clearly preserves
$\T_i$, $\scaf^N_{\T_i}$,
and each $C \in \acol$ on $\T_i$
along with the corresponding $\ptcol^N_{C}$. 
Since $\refl_L$ preserves $\T_i$
and each $C \in \acol$ on $\T_i$,
it in fact preserves $\{\T_i\}_{i=1}^3$
and $\acol$.
Since $\refl_L$ preserves
$\scaf^N_{\T_i}$ and $\ptcol^N_{C}$
for each $C \in \acol$ on $\T_i$, 
it in fact preserves all of $\scaf^N$
and all of $\ptcol^N$.

Now suppose $\Gamma \in \rcol$.
Then $\Gamma$ lies on some $\T_j$
and is equidistant from a pair of circles
in $\acol$ on $\T_j$.
In particular $\refl_\Gamma$
preserves $\T_j$, $\scaf^N_{\T_j}$,
and the collection of those circles in $\acol$ on $\T_j$.
It then follows easily that
$\refl_\Gamma$ preserves
$\{\T_i\}_{i=1}^3$, $\acol$,
all of $\scaf^N$, and all of $\ptcol^N$.

Since $\refl_L$ and $\refl_\Gamma$
preserve $\acol$,
they also preserve $\{T_C\}_{C \in \acol}$.
Since they preserve $\ptcol^N$,
they also preserve $\scol^N$.
Since they preserve $\{T_C\}_{C \in \acol}$,
$\scol^N$, and $\{\T_i\}_{i=1}^3$,
they preserve $\prcol^N$ too.
\end{proof}

For the reader's reference
we write down the generators
originating from $\rcol$:
\begin{equation}
\begin{gathered}
  \refl_{\Gamma_{12}}
    \begin{bmatrix}w_1 \\ w_2\end{bmatrix}
  =
  \frac{1}{\sqrt{2}}
    \begin{bmatrix}
      1 & -\ii \\ \ii & -1
    \end{bmatrix}
    \begin{bmatrix}w_1 \\ w_2\end{bmatrix},
  \qquad
  \refl_{\Gamma_{1^\perp 2^\perp}}
    \begin{bmatrix}w_1 \\ w_2\end{bmatrix}
  =
  \frac{1}{\sqrt{2}}
    \begin{bmatrix}
      -1 & \ii \\ -\ii & 1
    \end{bmatrix}
    \begin{bmatrix}w_1 \\ w_2\end{bmatrix},
  \\
  \refl_{\Gamma_{12^\perp}}
    \begin{bmatrix}w_1 \\ w_2\end{bmatrix}
  =
  \frac{1}{\sqrt{2}}
    \begin{bmatrix}
      1 & \ii \\ -\ii & -1
    \end{bmatrix}
    \begin{bmatrix}w_1 \\ w_2\end{bmatrix},
  \qquad
  \refl_{\Gamma_{1^\perp 2}}
    \begin{bmatrix}w_1 \\ w_2\end{bmatrix}
  =
  \frac{1}{\sqrt{2}}
    \begin{bmatrix}
      -1 & -\ii \\ \ii & 1
    \end{bmatrix}
    \begin{bmatrix}w_1 \\ w_2\end{bmatrix},
  \\
  \refl_{\Gamma_{13}}
    \begin{bmatrix}w_1 \\ w_2\end{bmatrix}
  =
  \frac{1}{\sqrt{2}}
    \begin{bmatrix}
      1 & 1 \\ 1 & -1
    \end{bmatrix}
    \begin{bmatrix}w_1 \\ w_2\end{bmatrix},
  \qquad
  \refl_{\Gamma_{1^\perp 3^\perp}}
    \begin{bmatrix}w_1 \\ w_2\end{bmatrix}
  =
  \frac{1}{\sqrt{2}}
    \begin{bmatrix}
      -1 & -1 \\ -1 & 1
    \end{bmatrix}
    \begin{bmatrix}w_1 \\ w_2\end{bmatrix},
  \\
  \refl_{\Gamma_{13^\perp}}
    \begin{bmatrix}w_1 \\ w_2\end{bmatrix}
  =
  \frac{1}{\sqrt{2}}
    \begin{bmatrix}
      1 & -1 \\ -1 & -1
    \end{bmatrix}
    \begin{bmatrix}w_1 \\ w_2\end{bmatrix},
  \qquad
  \refl_{\Gamma_{1^\perp 3}}
    \begin{bmatrix}w_1 \\ w_2\end{bmatrix}
  =
  \frac{1}{\sqrt{2}}
    \begin{bmatrix}
      -1 & 1 \\ 1 & 1
    \end{bmatrix}
    \begin{bmatrix}w_1 \\ w_2\end{bmatrix},
  \\
  \refl_{\Gamma_{23}}
    \begin{bmatrix}w_1 \\ w_2\end{bmatrix}
  =
  \begin{bmatrix}
    0 & e^{-\ii \pi/4} \\ e^{\ii \pi/4} & 0
  \end{bmatrix}
  \begin{bmatrix}w_1 \\ w_2\end{bmatrix},
\qquad
\refl_{\Gamma_{2^\perp 3^\perp}}
    \begin{bmatrix}w_1 \\ w_2\end{bmatrix}
  =
  \begin{bmatrix}
    0 & -e^{-\ii \pi/4} \\ -e^{\ii \pi/4} & 0
  \end{bmatrix}
  \begin{bmatrix}w_1 \\ w_2\end{bmatrix},
\\
\refl_{\Gamma_{23^\perp}}
    \begin{bmatrix}w_1 \\ w_2\end{bmatrix}
  =
  \begin{bmatrix}
    0 & -e^{\ii \pi/4} \\ -e^{-\ii \pi/4} & 0
  \end{bmatrix}
  \begin{bmatrix}w_1 \\ w_2\end{bmatrix},
\qquad
\refl_{\Gamma_{2^\perp 3}}
    \begin{bmatrix}w_1 \\ w_2\end{bmatrix}
  =
  \begin{bmatrix}
    0 & e^{\ii \pi/4} \\ e^{-\ii \pi/4} & 0
  \end{bmatrix}
  \begin{bmatrix}w_1 \\ w_2\end{bmatrix}.
\end{gathered}
\end{equation}

\subsection*{Colorings}
By a coloring of $\prcol^N$
we mean a function
$c: \prcol^N \to \{0,1\}$.
Given any $\gamma \in \scaf^N \cup \rcol$,
we say that a coloring $c$
is consistent along $\gamma$
if for any $\Omega_1,\Omega_2 \in \prcol^N$
sharing an edge on a circle
$\gamma \in \rcol \cup \scaf^N$
we have
\begin{equation}
  c(\Omega_1)=c(\Omega_2)
  \; \Leftrightarrow \;
  \refl_\gamma(\Omega_1)=\Omega_2.
\end{equation}
(Recall from
\ref{prisms_for_three_tori}.\ref{four_prisms}
that exactly four members of $\prcol^N$
meet along a given edge
on a given $\gamma \in \rcol \cup \scaf^N$,
and by \ref{action_on_collections}
these four are exchanged in pairs
under $\refl_\gamma$.)
We call $c$ consistent if
it is consistent along every
$\gamma \in \scaf^N \cup \rcol$.
Note that a consistent coloring,
if one exists,
is uniquely determine
by its value on any single
member of $\prcol^N$;
thus, if there is one consistent coloring,
then there are precisely two.
Note finally that
if $\rot \in O(4)$
preserves $\prcol^N$
and $c$ is a consistent coloring of $\prcol^N$,
then
$c \circ \rot$ is also
a consistent coloring of $\prcol^N$.

\begin{prop}[Colorings and orbits]
\label{colorings_and_orbits}
Let $N \in 4\Z$ be positive.
The collection $\prcol^N$ admits
a consistent coloring $c$
if and only if
$\prcol^N$ has precisely two orbits under $\grp^N$,
namely the preimages under $c$
of $0$ and $1$.
\end{prop}

\begin{proof}
Suppose $c$ is a consistent coloring of $\prcol^N$.
We observed in \ref{action_on_collections}
that $\prcol^N$ has at most two orbits under $\grp^N$,
and it is clear that
for each $\chi \in \{0,1\}$
there is a single orbit containing
$c^{-1}(\{\chi\})$.
To complete the proof in this direction
($\Rightarrow$)
it suffices to take any
$\gamma \in \scaf^N \cup \rcol$
and check that
$\refl_\gamma$ preserves $c$,
but $c \circ \refl_\gamma$
is a consistent coloring
which agrees with $c$ on all prisms
having an edge on $\gamma$
and therefore indeed agrees with $c$ globally.
The converse direction,
which we will not need,
follows even more directly from the definitions.
\end{proof}

\begin{prop}[Consistency]
\label{consistency}
Let $N \in 4+8\Z$ be positive.
Then $\prcol^N$ admits a consistent coloring.
\end{prop}

\begin{proof}
We define a coloring $c$ and confirm its consistency.
First we define 
$c|_{\prcol^N_{\acirc{1}}}$ by
\begin{equation}
  c(\Omega^N_{j\ell})
  =
  (j+\ell) \bmod 2.
\end{equation}
Next, for each $i \in \{2,3,2^\perp,3^\perp\}$
we set
\begin{equation}
  c|_{\prcol^N_{\acirc{j}}}
  =
  c|_{\prcol^N_{\acirc{1}}} \circ \refl_{\Gamma_{1j}}.
\end{equation}
We complete the definition of $c$
by setting
\begin{equation}
  c|_{\prcol^N_{\acirc{1}^\perp}}
  =
  c|_{\prcol^N_{\acirc{3}}} \circ \refl_{\Gamma_{1^\perp3}}.
\end{equation}

It is obvious from the construction of $c$
that it is consistent along
each circle in
\begin{equation}
  \scaf^N
  \cup
  \{
    \Gamma_{12}, \Gamma_{12^\perp},
    \Gamma_{13}, \Gamma_{13^\perp},
    \Gamma_{1^\perp3}
  \},
\end{equation}
so it remains to check consistency along
the seven circles
\begin{equation}
  \Gamma_{23}, \Gamma_{23^\perp},
  \Gamma_{2^\perp3},\Gamma_{2^\perp3^\perp},
  \Gamma_{1^\perp2}, \Gamma_{1^\perp2^\perp},
  \Gamma_{1^\perp3^\perp}.
\end{equation}
For $i \in \{2,2^\perp\}$ and $j \in \{3,3^\perp\}$
to verify consistency along
$\Gamma_{ij}$
(accounting for the first four circles
in this last list)
it suffices to check that
\begin{equation}
\label{first_four}
  \refl_{\Gamma_{1j}}
    \refl_{\Gamma_{ij}}
    \refl_{\Gamma_{1i}}
    \omega^N_{00}
  \in
  \{\omega^N_{k\ell}\}_{k+\ell \in 2\Z}
  \quad
  (i \in \{2,2^\perp\}, j \in \{3,3^\perp\}),
\end{equation}
while for $j \in \{2,2^\perp,3^\perp\}$
to verify consistency along
$\Gamma_{1^\perp j}$
(accounting for the final three circles)
it suffices to check that
\begin{equation}
\label{last_three}
  \refl_{\Gamma_{1j}}
    \refl_{\Gamma_{1^\perp j}}
    \refl_{\Gamma_{1^\perp3}} 
    \refl_{\Gamma_{13}}\omega^N_{00}
  \in
  \{\omega^N_{k\ell}\}_{k+\ell \in 2\Z}
  \quad
  (j \in \{2,2^\perp,3^\perp\}).
\end{equation}

We compute
\begin{equation}
\begin{gathered}
  \refl_{\Gamma_{12}}
    \refl_{\Gamma_{1^\perp 2}}
    \refl_{\Gamma_{1^\perp3}} 
    \refl_{\Gamma_{13}}
      \begin{bmatrix}w_1\\w_2\end{bmatrix}
  =
  \begin{bmatrix}
    -\ii & 0 \\ 0 & \ii
  \end{bmatrix}
  \begin{bmatrix}w_1\\w_2\end{bmatrix},
  \\
  \refl_{\Gamma_{12^\perp}}
    \refl_{\Gamma_{1^\perp 2^\perp}}
    \refl_{\Gamma_{1^\perp3}} 
    \refl_{\Gamma_{13}}
      \begin{bmatrix}w_1\\w_2\end{bmatrix}
  =
  \begin{bmatrix}
    \ii & 0 \\ 0 & -\ii 
  \end{bmatrix}
  \begin{bmatrix}w_1\\w_2\end{bmatrix},
  \\
  \refl_{\Gamma_{13^\perp}}
    \refl_{\Gamma_{1^\perp 3^\perp}}
    \refl_{\Gamma_{1^\perp3}} 
    \refl_{\Gamma_{13}}
      \begin{bmatrix}w_1\\w_2\end{bmatrix}
  =
  \begin{bmatrix}
    -1 & 0 \\ 0 & -1
  \end{bmatrix}
  \begin{bmatrix}w_1\\w_2\end{bmatrix},
\end{gathered}
\end{equation}
whence follows \ref{last_three},
using the fact that $N \in 4\Z$,
and we compute
\begin{equation}
\label{triple_product}
\begin{gathered}
  \refl_{\Gamma_{13}}\refl_{\Gamma_{23}}\refl_{\Gamma_{12}}
    \begin{bmatrix}w_1\\w_2\end{bmatrix}
  =
  \refl_{\Gamma_{13^\perp}}
    \refl_{\Gamma_{2^\perp 3^\perp}}
    \refl_{\Gamma_{12^\perp}}
    \begin{bmatrix}w_1\\w_2\end{bmatrix}
  =
  \begin{bmatrix}
    e^{\ii \pi/4} & 0 \\ 0 & -e^{-\ii \pi/4}
  \end{bmatrix}
    \begin{bmatrix}w_1\\w_2\end{bmatrix},
  \\
  \refl_{\Gamma_{13^\perp}}
    \refl_{\Gamma_{23^\perp}}\refl_{\Gamma_{12}}
    \begin{bmatrix}w_1\\w_2\end{bmatrix}
  =
  \refl_{\Gamma_{13}}
    \refl_{\Gamma_{2^\perp 3}}\refl_{\Gamma_{12^\perp}}
    \begin{bmatrix}w_1\\w_2\end{bmatrix}
  =
  \begin{bmatrix}
    e^{-\ii \pi/4} & 0 \\ 0 & -e^{\ii \pi/4}
  \end{bmatrix}
    \begin{bmatrix}w_1\\w_2\end{bmatrix},
\end{gathered}
\end{equation}
whence follows \ref{first_four},
using the assumption that $N \in 4+8\Z$.
\end{proof}

We can now describe $\grp^N$ in more detail.
Items \ref{two_orbits} and \ref{axial_syms_belong} below
are important for the construction,
and we mention items
\ref{turn_and_shift}
and
\ref{isoclinic_along_scaf_circ}
because they are helpful in relating
the surfaces constructed here
with certain surfaces constructed in \cite{KWtordesing};
we include the remaining items
simply to offer a fuller picture of $\grp^N$.

\begin{prop}[Some further properties of $\grp^N$]
\label{group_description}
Let $N \in 4+8\Z$ be positive.
Then
\begin{enumerate}[(i)]
  \item \label{two_orbits}
        the action of $\grp^N$
        on $\prcol^N$
        has precisely two orbits;

  \item  \label{axial_syms_belong}
  for each $\Omega \in \prcol^N$
  we have $\refl_{L_\Omega} \in \grp^N$
  (recalling \ref{prisms_for_three_tori}.\ref{axial_sym});

  \item \label{order}
  $\grp^N$ has order $48N$;

  \item \label{octahedral}
  the group homomorphism
  $\pghopf|_{\grp^N} \to O(3)$
  (recalling \ref{pghopf_declaration}
  and \ref{pghopf_definition})
  has kernel the cyclic group
  \begin{equation*}
    \ker \pghopf|_{\grp^N}
    =
    \Z_N
    :=
    \left\langle
      \rot_{\acirc{1}}^{2\pi/N}
      \rot_{\acirc{1}^\perp}^{2\pi/N}
    \right\rangle
  \end{equation*}
  and image the octahedral group
  (of order $48$)
  \begin{equation*}
    \im \pghopf|_{\grp^N}
    =
    O_h
    :=
    \{
      \mathsf{r} \in O(3)
      \; : \;
      \mathsf{r}
        ([-1,1]^3)
        =[-1,1]^3
    \}.
  \end{equation*}

  \item \label{stabilizer}
  the stabilizer
  $
    \stab_{\grp^N}(T_{\acirc{1}})
    :=
    \{
      \rot \in \grp^N
      \; : \;
      \rot T_{\acirc{1}} = T_{\acirc{1}}
    \}
  $
  of $T_{\acirc{1}}$ in $\grp^N$
  has order $8N$
  and is generated by
  $\refl_{L_{\Omega_{00}}}$,
  $\refl_{L_2}$, and $\refl_{L_3}$
  for any single 
  $L_2 \in \scaf^N_{\T_2}$
  and any single $L_3 \in \scaf^N_{\T_3}$;

  \item \label{stabilizer_generates}
  $\grp^N$ can be obtained
  from $\stab_{\grp^N}(T_{\acirc{1}})$
  by adjoining
  $\refl_{\Gamma_{12}}$;

  \item \label{L_subgroup}
  the subgroup
  $
    \grp_{\scaf^N}
    :=
    \langle
      \refl_L
    \rangle_{L \in \scaf^N}
  $
  of $\grp^N$ has order $8N$;

  \item \label{Gamma_subgroup}
  the subgroup
  $
    \grp_\Gamma
    :=
    \langle
      \refl_\Gamma
    \rangle_{\Gamma \in \rcol}
  $
  of $\grp^N$ has order $48$;

  \item \label{turn_and_shift}
  $
    \rot_{\acirc{1}}^{\pi/2}
      \left(
        \rot_{\acirc{1}}^{\pi/N}
          \rot_{\acirc{1}^\perp}^{\pi/N}
      \right)
    \in
    \grp^N
  $; and
  
  \item \label{isoclinic_along_scaf_circ}
  for any $L \in \scaf^N$
  we have
  $
    \rot_L^{\pi/4}\rot_{L^\perp}^{\pi/4}
    \in
    \grp^N   
  $,
  where the relative orientation
  of $L$ and $L^\perp$
  is the one such that this product
  preserves the $\T_i$
  containing $L \cup L^\perp$.
\end{enumerate}
\end{prop}

\begin{proof}
Item \ref{two_orbits}
is immediate
from \ref{colorings_and_orbits}
and \ref{consistency}.
Given that $\refl_L \in \grp^N$
for every $L \in \scaf^N$,
item \ref{axial_syms_belong}
is equivalent to
\begin{equation}
\label{to_check}
  \rot_{\acirc{1}}^{\frac{\pi}{N}+\frac{\pi}{2}}
      \rot_{\acirc{1}^\perp}^{\frac{\pi}{N}}
  \in
  \grp^N,
\end{equation}
which is in fact identical to claim
\ref{turn_and_shift}.
However, it is clear that
\begin{equation}
\label{iso}
  \rot_{\acirc{1}}^{2\pi/N}\rot_{\acirc{1}^\perp}^{2\pi/N}
  \in
  \langle \refl_L \rangle_{L \in \scaf^N_{\T_2}}
  <
  \grp^N,
\end{equation}
and we see from \ref{triple_product}
that
\begin{equation}
  \rot_{\acirc{1}}^{\pi/4}\rot_{\acirc{1}^\perp}^{-\pi/4}
  \in
  \grp^N,
\end{equation}
so,
because our assumption on $N$
ensures that $\frac{N+4}{8}$ is a positive integer,
we also have
\begin{equation}
  \rot_{\acirc{1}}^{\frac{\pi}{N}+\frac{\pi}{2}} 
    \rot_{\acirc{1}^\perp}^{\frac{\pi}{N}}
  =
  \rot_{\acirc{1}}^{\pi/4}\rot_{\acirc{1}^\perp}^{-\pi/4}
    \left(
      \rot_{\acirc{1}}^{\frac{2\pi}{N}}
      \rot_{\acirc{1}^\perp}^{\frac{2\pi}{N}}
    \right)^{\frac{N+4}{8}}
  \in
  \grp^N,
\end{equation}
verifying \ref{to_check},
and so items
\ref{axial_syms_belong}
and
\ref{turn_and_shift}.

In view of item
\ref{two_orbits}
(and the fact that $\prcol^N$ has cardinality $48N$)
the $\grp^N$ orbit of any $\Omega \in \prcol^N$
consists of $24N$ prisms,
while
item \ref{axial_syms_belong}
and
\ref{prisms_for_three_tori}.\ref{axial_sym}
imply that the stabilizer in $\grp^N$
of any $\Omega \in \prcol^N$
has order two,
confirming item \ref{order}.

For \ref{octahedral}
it is clear that each $\refl_\gamma$
with $\gamma \in \scaf^N \cup \rcol$
preserves the set of fibers of $\phopf$,
and so the map in the statement
is indeed well-defined.
Since $\grp^N$ preserves $\acol$,
it is also clear that
$\pghopf(\grp^N)$ is a subgroup of $O_h$.
Investigating more closely,
we observe that
$\pghopf(\{\refl_\Gamma\}_{\Gamma \in \rcol})$
consists of the six reflections
through each of the six lines
that pass through the origin
and a pair of midpoints of edges
of the cube $[-1,1]^3$.
These last six symmetries
generate $O_h \cap SO(3)$,
and so
by adjoining
any single $\pghopf(\refl_L)$
with  $L \in \scaf^N$
we obtain the full group $O_h$,
verifying the assertion
about the image of the map in \ref{octahedral}.
On the other hand,
from \ref{iso}
we see that the kernel contains
$\Z_N$,
but the orders of $\Z_N$, $\grp^N$, and $O_h$
are respectively
$N$, $48N$, and $48$,
completing the proof of \ref{octahedral}.

We proceed to item \ref{stabilizer}.
Using item \ref{two_orbits} again,
we can verify
that the stabilizer
of $T_{\acirc{1}}$ in $\grp^N$
acting on $\prcol^N_{\acirc{1}}$
has two orbits,
each of cardinality $4N$.
On the other hand,
for any $\Omega \in \prcol^N_{\acirc{1}}$
the reflection
$\refl_{L_\Omega}$
preserves $T_{\acirc{1}}$,
so with item \ref{axial_syms_belong}
and \ref{prisms_for_three_tori}.\ref{axial_sym},
we verify the asserted order
in item \ref{stabilizer}.
The proof of \ref{stabilizer}
is then completed
by observing that,
for any $L_2,L_3$
as in the second half of the statement
of \ref{stabilizer},
the group
$
  \langle
    \refl_{L_{\Omega_{00}}},
    \refl_{L_2},
    \refl_{L_3}
  \rangle
$
also has order $8N$
and preserves $T_{\acirc{1}}$.

Now, for \ref{stabilizer_generates},
let $\subgrp$
be the group generated
by the $\refl_{\Gamma_{12}}$
and the stabilizer
of $T_{\acirc{1}}$ in $\grp^N$.
To prove \ref{stabilizer_generates}
it suffices to check
that $\subgrp$
includes $\refl_\gamma$
for every $\gamma \in \Gamma \cup \scaf^N$.
Clearly $\refl_L \in \subgrp$
for every
$L \in \scaf^N_{\T_2} \cup \scaf^N_{\T_3}$,
but
$\refl_{\Gamma_{12}}\scaf^N_{\T_2}=\scaf^N_{\T_1}$,
so indeed
$
  \langle \refl_L \rangle_{L \in \scaf^N}
  <
  \subgrp
$.
Since
$
  \refl_{\acirc{1}}
  \in
  \langle 
    \refl_L
  \rangle_{L \in \scaf^N_{\T_2} \cup \scaf^N_{\T_3}}
  <
  \subgrp
$
and $\refl_{\Gamma_{12}} \in \subgrp$,
we get also
$
  \refl_{\Gamma_{12\perp}}, ~
  \refl_{\acirc{2}}
  \in
  \subgrp
$,
and then similarly
$
  \refl_{\Gamma_{1^\perp2}}, ~
  \refl_{\acirc{1}^\perp}, ~
  \refl_{\acirc{2}^\perp}, ~
  \refl_{\Gamma_{1^\perp2^\perp}}
  \in
  \subgrp
$.
In particular we have
$\refl_\Gamma \in \subgrp$
for every $\Gamma \in \rcol$
such that $\Gamma \subset \T_3$.
Using $\refl_{L_{\Omega_{00}}} \in \subgrp$,
$\refl_{L_{\Omega_{00}}}\T_3 = \T_2$,
and $\refl_{L_{\Omega_{00}}}\T_1 = \T_1$,
we then get
$\refl_\Gamma \in \subgrp$
for every $\Gamma \in \rcol$
such that $\Gamma \subset \T_2$.
Finally,
since $\refl_{\Gamma_{12}}$
preserves $\T_3$
and exchanges $\T_1$ and $\T_2$,
we get in turn
$\refl_\Gamma \in \subgrp$
for every $\Gamma \in \rcol$
such that $\Gamma \subset \T_1$,
completing the proof of
\ref{stabilizer_generates}.

Turning to \ref{L_subgroup},
it is easy to see that
the orbit under $\grp_{\scaf^N}$
of any $\Omega \in \prcol^N_{\acirc{1}}$
includes at least $4N$ members of
$\prcol^N_{\acirc{1}}$
and at least $4N$ members of
$\prcol^N_{\acirc{1}^\perp}$.
In view of item \ref{two_orbits}
this orbit does not contain any other members
of $\prcol^N_{\acirc{1}} \cup \prcol^N_{\acirc{1}^\perp}$,
and since each $\refl_L$ with $L \in \scaf^N$
preserves each $\T_i$,
the orbit of $\Omega$ under $\grp_{\scaf^N}$
in fact consists of
exactly these $8N$ prisms
and $\grp_{\scaf^N}$
does not include $L_\Omega$,
ending the proof of \ref{L_subgroup}.

One can give a proof of \ref{Gamma_subgroup}
similar to that of
\ref{consistency}.
In outline,
first we observe
that every $\refl_\Gamma$
with $\Gamma \in \rcol$ is unitary
(preserves the orientations
on the fibers of $\phopf$),
while $\refl_{L_{\Omega_{00}}}$
is antiunitary
(reverses the orientations
on the fibers of $\phopf$),
and so it suffices to check that the orbit
of $\Omega_{00}$ under $\grp_\Gamma$
has cardinality $48$.
We may assume (in this paragraph) $N=4$.
Next we define a map
$c: \prcol^4 \to \{0,1,2,3\}$
by first defining
$c|_{\prcol^4_{\acirc{1}}}$
as
$
  c(\Omega_{ij})
  :=
  (j-i) \bmod 4
$
and then extending to all of $\prcol^4$
as in
the proof of \ref{consistency}.
Then one can check,
much as in the proof of
\ref{consistency}
and with some modifications
to the discussion immediately preceding it,
that $\grp_\Gamma$ preserves $c$
and then that it has order $48$.
We leave the details to the reader.

It remains only to check
item \ref{isoclinic_along_scaf_circ}.
In fact it suffices to consider $L \in \scaf^N_{\T_3}$.
Then $L$ and $L^\perp$
orthogonally intersect both
$\acirc{1}$ and $\Gamma_{12}$;
furthermore
$\acirc{1}, \Gamma_{12} \subset \T_3$,
and
$d^{\Sph^3}_{\acirc{1}}$
takes the constant value $\frac{\pi}{8}$
on $\Gamma_{12}$.
Accordingly,
\begin{equation}
  \rot_{L}^{\pi/4}\rot_{L^\perp}^{\pi/4}
  =
  \refl_{\Gamma_{12}}\refl_{\acirc{1}}
  \in
  \grp^N,
\end{equation}
where each of $L$ and $L^\perp$
is oriented from either point of
intersection with $\acirc{1}$
toward the closest point (to this one)
of intersection with $\Gamma_{12}$,
completing the proof.
\end{proof}

\subsection*{Proof of Theorem \ref{thm:three_tori}}

Let $N=4+8m$,
and let $Q$ be the hexagon
in $\partial \Omega_{00}$
which is the union of the edges
contained in $\bigcup \rcol \cup \bigcup \scaf^N$.
By \ref{prisms_for_three_tori}.\ref{congruence}
and \ref{unique_disc}
there exists a unique connected minimal surface
$D \subset \Omega$ with $\partial \Omega=Q$,
and $D$ is an embedded disc.
By the maximum principle
$
  D \cap \partial \Omega_{00}
  =
  \partial D
$.
Then from the uniqueness of $D$
and the fact
(\ref{group_description}.\ref{two_orbits})
that the action
of $\grp^N$ on $\prcol^N$
has two orbits
it follows that the set
$M_m := \grp^N(D)$
is a closed embedded minimal surface.

From the geometry of $\Omega_{00}$
(\ref{prisms_for_three_tori}
and \ref{prism_geometry})
we see that every vertex angle of $Q$ is right.
Using the Gauss-Bonnet theorem
and the fact that $D$ is a disc
with piecewise geodesic boundary,
we conclude that $D$ has total curvature
$-\pi$.
Thus $M_m$ has total curvature
$-24N\pi$,
and another application
of the Gauss-Bonnet theorem
then establishes that
$M_m$ has genus
$6N+1 = 48m + 25$.
By construction
$M_m$ contains
$\bigcup \rcol \cup \scaf^N$
and is invariant under $\grp^N$.
This completes the proof
of the existence portion
of \ref{thm:three_tori}.

The uniqueness assertion
follows immediately
from the reflection principle
and \ref{unique_disc}.
To prove that $M_m$
is congruent,
for sufficiently large $m$,
to the surface constructed
in \cite{KWtordesing}*{Theorem 7.1}
as a small graphical perturbation
of the initial surface
$N(2, 2m+1, 1, 1, 1, 0, 1)$
specified in \cite{KWtordesing}*{Definition 4.13}
it suffices---as explained at the end
of the proof of Section \ref{sec:uniqueness},
in the proof of
Theorem \ref{odd_cs_uniqueness}---to verify
that (modulo congruence)
this initial surface
is invariant under $\grp^N$
and that its intersection
with the interior of $\Omega_{00}$
is nonempty
and has boundary $\partial D$. 

These last properties are easily checked
by inspection of
\cite{KWtordesing}*{4.15}
(reviewing also the supporting definitions),
with the aid of
\ref{group_description}.\ref{turn_and_shift}
and \ref{group_description}.\ref{isoclinic_along_scaf_circ}.
For this it is important to note an isolated misprint
in \cite{KWtordesing}*{4.15}:
the $\Sph^3$ isometry
\begin{equation}
\label{misprint}
  \rot_{\underline{C}}^{\pi/kmn'_{(-1)^j}}
  \;
  \rot_{\underline{C}^\perp}^{\pi/kmn'_{(-1)^j}}
\end{equation}
appearing in parentheses
in the middle line of \cite{KWtordesing}*{4.15}
is missing factors of $2$
in the denominators of both exponents
and should instead read
\begin{equation}
\label{corrected}
  \rot_{\underline{C}}^{\pi/2kmn'_{(-1)^j}}
  \;
  \rot_{\underline{C}^\perp}^{\pi/2kmn'_{(-1)^j}},
\end{equation}
which
in the notation of the present article
is (up to an overall congruence)
\begin{equation}
  \rot_{\acirc{1}}^{\pi/N}
  \; 
  \rot_{\acirc{1}^\perp}^{\pi/N}.
\end{equation}
The significance of this correction is that
the isometry \ref{misprint}
amounts to shifting the corresponding
singly periodic Scherk surface---glued in
along the great circle indexed by $j$---by
a full period, rather than a half period,
counter to the intended effect
of producing distinct initial surfaces,
which is instead accomplished by \ref{corrected},
corresponding to a half-period shift.
\qed

\end{document}